\documentclass{article}
\usepackage[margin=1in]{geometry}
\usepackage{xcolor}
\usepackage{hyperref}
\usepackage{bbm}
\usepackage{graphicx}
\usepackage[utf8]{inputenc}
\usepackage[T1]{fontenc}
\usepackage{amsmath,amsthm,amssymb}

\newtheorem{lemma}{Lemma}
\newtheorem{theorem}{Theorem}
\newtheorem{remark}{Remark}
\newtheorem{proposition}{Proposition}

\newtheorem{definition}{Definition}
\newtheorem{corollary}{Corollary}

\def\<#1>{\left\langle #1 \right\rangle}

\newcommand{\w}{\ensuremath{\mathrm{w}}}

\newcommand{\z}{\ensuremath{\mathrm{z}}}

\newcommand{\bP}{\ensuremath{p}}
\newcommand{\bQ}{\ensuremath{q}}

\newcommand{\FP}{\ensuremath{\Pi}}
\newcommand{\axisDir}{\ensuremath{\mathrm{A}}}
\newcommand{\angleFP}{\ensuremath{\theta}}
\newcommand{\sA}{\ensuremath{\alpha}}
\newcommand{\sB}{\ensuremath{\beta}}

\newcommand{\sPhi}{\ensuremath{\varphi}}

\newcommand{\V}{\ensuremath{\mathcal{V}}}

\newcommand{\ci}{\ensuremath{\mathrm{i}\,}}
\newcommand{\qi}{\ensuremath{\mathbbm{i}}}
\newcommand{\qj}{\ensuremath{\mathbbm{j}}}
\newcommand{\qk}{\ensuremath{\mathbbm{k}}}

\newcommand{\N}{\ensuremath{\mathbbm{N}}}

\newcommand{\R}{\ensuremath{\mathbbm{R}}}
\newcommand{\C}{\ensuremath{\mathbbm{C}}}

\renewcommand{\H}{\ensuremath{\mathbbm{H}}}
\renewcommand{\S}{\ensuremath{\mathbbm{S}}}

\renewcommand{\Re}{\ensuremath{\mathrm{Re}}}
\newcommand{\Imc}{\ensuremath{\mathrm{Im}_{\C}}}
\newcommand{\Imq}{\ensuremath{\mathrm{Im}_{\H}}}

\graphicspath{{figures/}}
\author{Alexander I. Bobenko, Tim Hoffmann, Andrew O. Sageman-Furnas}
\date{\today}
\title{Isothermic tori with one family of planar curvature lines and area constrained hyperbolic elastica}

\AtEndDocument{\bigskip{\footnotesize%
		\noindent Alexander I. Bobenko: \texttt{bobenko@math.tu-berlin.de} \\
		\textsc{Institute of Mathematics, Secretary MA 8-4, \\ Technical University of Berlin, 10623 Berlin, Germany} \par 
		\addvspace{\medskipamount}
		\noindent Tim Hoffmann: \texttt{tim.hoffmann@ma.tum.de} \\
		\textsc{Department of Mathematics, \\ Technical University of Munich, 85748 Garching, Germany} \par
		\addvspace{\medskipamount}
		\noindent Andrew O. Sageman-Furnas: \texttt{asagema@ncsu.edu} \\
		\textsc{Department of Mathematics, \\ North Carolina State University, Raleigh, NC 27695, USA} \par
}}
\hypersetup{
 pdfauthor={},
 pdftitle={},
 pdfkeywords={},
 pdfsubject={},
 pdfcreator={Emacs 25.1.1 (Org mode 8.3.4)}, 
 pdflang={English},
 final
}

\begin{document}
\maketitle

\begin{abstract}
In 1883, Darboux gave a local classification of isothermic surfaces with one family of planar curvature lines using complex analytic methods. His choice of real reduction cannot contain tori.

We classify isothermic tori with one family of planar curvature lines. They are found in the second real reduction of Darboux's description. We give explicit theta function formulas for the family of plane curves. These curves are particular area constrained hyperbolic elastica. With a Euclidean gauge, the Euler--Lagrange equation is lower order than expected.

In our companion paper~\cite{short-compact-bonnet} we use such isothermic tori to construct the first examples of compact Bonnet pairs: two isometric tori related by a mean curvature preserving isometry. They are also the first pair of isometric compact immersions that are analytic.
 
Additionally, we study the finite dimensional moduli space characterizing when the second family of curvature lines is spherical. Isothermic tori with planar and spherical curvature lines are natural generalizations of Wente constant mean curvature tori, discovered in 1986. Wente tori are recovered in a limit case of our formulas.
\end{abstract}

\tableofcontents


\section{Introduction}
\label{sec:intro}

Isothermic surfaces are prominent in differential geometry. The most well-known examples are minimal and constant mean curvature surfaces, quadrics, and surfaces of revolution. Isothermic surfaces are governed by an integrable system~\cite{CieliskiGoldsteinSym1995,CieslinskiKobus2017,BurstallHertrich-JerominPeditPinkall_CurvedFlats}, which provides a hierarchy of special classes~\cite{BurstallSantos2012}. Moreover, they are natural objects in conformal geometry~\cite{BurstallFerusLeschkePeditPinkall2002,Hertrich-Jeromin2003,Hertrich-JerominMussoNicolodi,MR4622235} and related geometric functionals like the Willmore energy~\cite{MR1795097,BohlePetersPinkall2008,Riviere2014,MR4291610}. They also have discrete analogs~\cite{BobenkoPinkall1996,MR2392888,LamPinkall2017,BurstallChoHertrich-JerominPemberRossman23} and generalizations to higher dimensions~\cite{Palmer88,Schief2001,Ma2005,Tojeiro2006,BurstallDonaldsonPeditPinkall2011,MR2927680,DajczerTojeiro2016}.

Our study of isothermic tori is motivated by the seemingly unrelated longstanding problem: do a metric and mean curvature function determine a unique compact surface in Euclidean three space?

In 1981, Lawson and Tribuzy proved that each metric and nonconstant mean curvature function gives at most two compact immersions~\cite{LawTri}. However, it was unknown if two such immersions exist. Two noncongruent immersed surfaces form a \emph{Bonnet pair} if they are related by a mean curvauture preserving isometry. The above question is sometimes called the global Bonnet pair problem. Kamberov--Pedit--Pinkall gave a local characterization of Bonnet pairs in terms of conformal transformations of isothermic surfaces~\cite{KPP}. So, one can ask if there exists a compact isothermic surface that gives rise to compact Bonnet pairs.

We resolve the global Bonnet problem using this approach in our companion paper~\cite{short-compact-bonnet}. We explicitly construct the first examples of noncongruent isometric tori with the same mean curvature. They are also the first pair of isometric compact immersions that are analytic. These isometric tori arise from an isothermic torus with one family of planar curvature lines. The classification of such isothermic tori is the focus of the present paper. Numerical experiments using a discrete differential geometry theory of Bonnet pairs~\cite{HSFW17,HSFW24} arising from discrete isothermic surfaces~\cite{BobenkoPinkall1996} revealed that smooth Bonnet pair tori may arise from smooth isothermic tori with one family of planar curvature lines.

In 1883, Darboux derived a local classification of isothermic surfaces with one family of planar curvature lines \cite{Darboux1883,Darboux1896}. He gave a geometric construction for such surfaces, together with a description for the family of planar curves in terms of theta functions. He has not considered the global problem, however, and his formulas do not include tori. The reason is that the reality conditions he imposed on his complex analytic investigations only lead to elliptic curves with rectangular period lattices. In Section~\ref{sec:isothermic-planar-u-local} we complete Darboux's local classification by considering the real reduction on elliptic curves with rhombic period lattices. In Section~\ref{sec:isothermic-planar-u-global} we prove that this missing rhombic case contains isothermic tori.

This missing case played an important role for the further development of the global theory of surfaces. Darboux's classification was partially motivated by the study of constant mean curvature surfaces and was very influential. Follow-up work showed that constant mean curvature surfaces are a limiting case of Darboux's formulas, but also only considered elliptic curves of rectangular type~\cite{Adam1893}. Nearly 100 years later, Wente tori were discovered. They were the first examples of compact constant mean curvature surfaces that are not round spheres~\cite{Wente1986}. Their existence sparked a flurry of new discoveries~\cite{PS1989,Kapouleas1990,KorevaarKusner93}. Wente tori have one family of planar curvature lines~\cite{Abresch1987,Walter1987} that is governed by an elliptic curve of rhombic type.

In contrast to the constant mean curvature case, isothermic surfaces with one family of planar curvature lines depend on a functional freedom. For an isothermic surface $f(u,v)$ with planar $u$-lines there exists a mapping of all planar $u$-lines into a common plane such that the family of planar curves $u \mapsto \gamma(u, \w)$ is holomorphic with respect to $u + \ci \w$ for a \emph{reparametrization function} $\w(v)$. Conversely, given a particular holomorphic family of planar curves $\gamma$, each choice of reparametrization function $\w(v)$ determines an isothermic surface with one family of planar curvature lines. The surface depends on the choice of reparametrization function $\w(v)$, which is the functional freedom.

In Section~\ref{sec:hyperbolic-elastica} we show these holomorphic families of plane curves are families of area constrained hyperbolic elastica. This exemplifies the relationship between isothermic surfaces with spherical curvature lines and constrained elastica in space forms~\cite{ChoPemberSzewieczek2021}. It also generalizes the observation that Willmore tori of revolution have profile curves that are hyperbolic elastica~\cite{MR0751827}. We remark that rotational linear Weingarten surfaces in space forms also have profile curves that are critical for a suitable curvature energy~\cite{MR4089078,montaldo2020existence}.

Our Euclidean gauge on these area constrained hyperbolic elastic curves has an unexpected outcome. In Theorem~\ref{thm:elasticaLowerOrder} we prove the Euler--Lagrange equation is first order in the Euclidean unit tangent vector. This is one order less than the usual Euler--Lagrange equation for hyperbolia elastica, which is first order in the hyperbolic curvature function~\cite{MR0772124,MR0844630,MR1314459,MR3210758}. 

Our main global results are in Section~\ref{sec:isothermic-planar-u-global}. In Theorem~\ref{thm:planarIsothermicCylinderFormulas}, we classify all isothermic cylinders with one family of closed planar curvature lines. The functional freedom $\w(v)$ remains. Moreover, their sphere inversion and Christoffel duality symmetries are given in Theorem~\ref{thm:planarIsothermicCylinderGeometry}. For an isothermic cylinder $f$ to close into a torus the reparametrization function $\w(v)$ must be periodic. In Theorem~\ref{thm:rationalityTorusClosing} we show that $f$ either closes after one period of $\w(v)$, or is generated by the rotation of a fundamental piece around an axis. So, $f$ is a torus when the rotation angle is a rational multiple of $2\pi$.

In Section~\ref{sec:isothermic-planar-u-spherical-v} we consider the local theory of isothermic surfaces with one family of planar curvature lines with the additional geometric restriction that the second family of curvature lines is spherical. We show that the moduli space is finite dimensional, and in Theorem~\ref{thm:sphericalEllipticCurve} prove that the spherical curvature lines are governed by a second elliptic curve. In Section~\ref{sec:isothermic-planar-u-spherical-v-global} we consider global aspects.

Wente tori are constant mean curvature surfaces with one family of planar and one family of spherical curvature lines. Isothermic tori with planar and spherical curvature lines are their natural generalizations. Motivated by the Bonnet pair problem, Bernstein~\cite{Bernstein2001} constructed new examples of isothermic tori with two families of spherical curvature lines, but these did not give rise to Bonnet pair tori. Our complete classification in the planar and spherical case gives important explicit formulas for closing conditions. In Proposition~\ref{prop:bigZPrimeAtOmegaMovingFrameCoefficients} we compute the axis containing the centers of the spherical curvature line spheres. In our companion paper, we show this gives analytic expressions for the closing conditions of the resulting Bonnet pair tori, allowing us to prove existence~\cite{short-compact-bonnet}.

Isothermic tori with one family of planar curvature lines therefore helped resolve multiple longstanding questions on global surface theory, so far. They give rise to the first examples of compacts Bonnet pairs, the first pair of analytic isometric compact immersions, and Wente tori were the first examples of compact constant mean curvature surfaces that are not round spheres. The explicit classification given here may prove useful in other global problems. Especially, since there is unexpected functional freedom in the construction and a relationship to hyperbolic geometry. 

\paragraph{Acknowledgements}
We thank Gudrun Szewieczek for helpful discussions about constrained elastica in space forms. This research was supported by the DFG Collaborative Research Center TRR 109 ``Discretization in Geometry and Dynamics''.


\section{Preliminaries}
\label{sec:preliminaries}
\subsection{Structure equations for isothermic surfaces}
\label{sec:diff-eq-surfaces}

Let ${\mathcal F}$ be a smooth orientable surface in 3-dimensional Euclidean space.
The Euclidean metric induces a metric $\Omega$ on this surface, which
in turn generates the complex structure of a Riemann surface $\mathcal R$.
 Under such a parametrization, which is called  conformal, the surface
$\mathcal F$ is given by an immersion
$$
f=(f_1, f_2, f_3) : {\mathcal R} \rightarrow \R^3,
$$
and the metric is conformal: $\Omega=e^{2h}\, d \z d\bar{\z}$, where $\z$ is
a local coordinate on $\mathcal R$. Denote by $u$ and $v$ its real and imaginary parts: $ \z =u+\ci v$.

The tangent vectors $f_u, f_v$  together with the unit normal $n:{\mathcal R}\to \mathbb{S}^2$ define a conformal moving frame on the surface:
\begin{align*}
\<f_u, f_u>=\<f_v, f_v>&=e^{2h},\
\<f_u,f_v>=0,& \<f_u, n>=\<f_v,n>=0,\ \<n,n>=1.&
\end{align*}

\begin{remark}
	The function $e^h$ plays a key role in our considerations. Even though this is formally the square root of the metric function $e^{2h}$, we often refer to $e^h$ as the metric and $h$ as the metric factor.
\end{remark}

A parametrization that is simultaneously conformal and curvature line is called {\em isothermic}. We refer to $f$ as an \emph{isothermic immersion}.
In this case the preimages of the curvature lines are the lines
$u={\rm const}$ and $v={\rm const}$ on the parameter domain, where
$\z =u+\ci v$ is a conformal coordinate. Equivalently, a parametrization is isothermic if it is conformal and $f_{uv}$ lies in the tangent plane, i.e.,
\begin{equation}   \label{eq:f_uv-isothermic}
f_{uv}\in {\rm span}\{ f_u, f_v \}.
\end{equation}
A surface  is called {\em isothermic} if it admits an isothermic parametrization.

The differential equations describing isothermic surfaces simplify in isothermic coordinates. The frame equations are as follows: 
\begin{eqnarray}
\left(\begin{array}{c}
f_u\\
f_v\\
n
\end{array}\right)_u=
\left(\begin{array}{ccc}
h_u & -h_v & k_1e^{2h}\\
h_v & h_u & 0\\
-k_1 & 0 & 0
\end{array}\right)
\left(\begin{array}{c}
f_u\\
f_v\\
n
\end{array}\right),
\label{eq:isothermicFrameEquationsDU}
\\
\left(\begin{array}{c}
f_u\\
f_v\\
n
\end{array}\right)_v=
\left(\begin{array}{ccc}
h_v & h_u &  0\\
-h_u & h_v & k_2e^{2h}\\
0 & -k_2 & 0
\end{array}\right)
\left(\begin{array}{c}
f_u\\
f_v\\
n
\end{array}\right).
\label{eq:isothermicFrameEquationsDV}
\end{eqnarray}
Here $k_1, k_2$ are the principal curvatures along the $u$ and $v$ curvature lines.

The first, second, and the third fundamental forms are given by
\begin{eqnarray}
\label{eq:fund.forms_isothermic}
\<df,df> & =& e^{2h} (du^2+dv^2), \nonumber\\
- \< df,dn >&=&e^{2h}(k_1du^2+k_2dv^2), \\
\<dn,dn>&=&e^{2h}(k_1^2 du^2+k_2^2 dv^2).\nonumber
\end{eqnarray}
The Gauss--Codazzi integrability conditions are
\begin{eqnarray}
& h_{uu}+h_{vv}+k_1k_2e^{2h}=0,  \label{eq:Gauss_isothermic}\\
& {k_2}_u=h_u(k_1-k_2),\quad {k_1}_v=h_v(k_2-k_1).\label{eq:Codazzi_isothermic}
\end{eqnarray}

\subsection{Christoffel dual isothermic surface}

Let $D\subset \C$ be a simply connected domain and $f:D\to \R^3$ be an isothermic immersion without umbilic points. Its differential is $df=f_u du+f_v dv$. An important property of an isothermic immersion is that the following form is closed.
\begin{equation}
\label{eq:isothermic_dual}
df^*:=e^{-2h}(f_u du -f_v dv).
\end{equation}
The corresponding immersion $f^*:D\to \R$, which is determined up to a translation, is also isothermic and called the {\em (Christoffel) dual isothermic surface}~\cite{Christoffel1867,JensenMussoNicolodi2016}. Relation \eqref{eq:isothermic_dual} is an involution.

Note that the dual isothermic surface is defined through one-forms and the periodicity properties of $f:{\mathcal R}\to \R^3$ are not respected.

\subsection{Quaternionic description of surfaces}
\label{sec:quaternionic-desc}

We construct and investigate surfaces in $\R^3$ by analytic methods.
For this purpose it is convenient to rewrite the immersion in terms of quaternions.
This quaternionic description is useful for studying general curves and surfaces in 3- and
4-space, and particular special classes of surfaces \cite{Bob_CMC1991, DPW, KPP, BurstallFerusLeschkePeditPinkall2002}.

Let us denote the algebra of quaternions by $\H$, the multiplicative quaternion
group by ${\H_*}={\H}\setminus\{0\}$, and their standard basis by $\{{\bf1}
,{\qi},{\qj},{\qk}\}$, where
\begin{align}
	{\qi}{\qj}={\qk},\
	{\qj}{\qk}={\qi},\
	{\qk}{\qi}={\qj},\
	\qi^2 = \qj^2 = \qk^2 = -1.                                      \label{eq:quaternion_algebra}
\end{align}
This basis can be represented for example by the matrices 
\begin{eqnarray*}
	{\bf1}=\left( \begin {array}{cc} 1&0\\0&1 \end{array}\right), \
\qi=\left( \begin {array}{cc} 0&-\ci \\-\ci &0 \end{array}\right),\ 
\qj=\left( \begin {array}{cc} 0&{-1}\\1&0 \end{array}\right),\
\qk=\left( \begin {array}{cc} -\ci &0\\0& \ci \end{array}\right).
\end{eqnarray*}

We identify
$\H$ with 4-dimensional Euclidean space
$$
q=q_0 {\bf 1}+ q_1 \qi+q_2 \qj+q_3 \qk
\longleftrightarrow \ q=(q_0, q_1, q_2, q_3)\in\R^4.
$$
The length of a quaternion is $|q|^2 = q \overline q$, where $\overline q =q_0 {\bf 1}- q_1 \qi-q_2 \qj-q_3 \qk$ is the conjugate of $q$. The inverse of $q \neq 0$ is $q^{-1} = \frac{\overline q}{|q|^2}.$ The sphere $\S^3\subset\R^4$ is naturally identified with the group of unitary
quaternions $\H_1=\mathrm{SU}(2)$.

Three dimensional Euclidean space $\R^3$ is identified with the space of imaginary quaternions
${\rm Im}\ {\H}$
\begin{eqnarray}
X =X_1\qi +X_2\qj +X_3\qk \in{\rm Im}\ {\H}\
\longleftrightarrow \ X=(X_1, X_2, X_3)\in\R^3.  \label{eq:R^3=ImH}
\end{eqnarray}
The conjugate of $X \in \R^3$ is $\overline X = -X$. The scalar and the cross products of vectors in terms of quaternions are
\begin{eqnarray} 
XY=-\<X,Y>+X \times Y, 	 \label{eq:products}
\end{eqnarray}
in particular
$$
[X,Y]=XY-YX=2 X\times Y.
$$
Throughout this article we will not distinguish quaternions, their matrix representation, and their vectors in $\R^3$. For example vectors $f$ and $n$ are also identified with imaginary quaternions, and we can also write the Christoffel dual one-form \eqref{eq:isothermic_dual} as $df^*=-f_u^{-1} du + f_v^{-1} dv.$

Moreover, we identify the space of complex numbers $\C$ with the span of ${\bf1}$ and $\qi$.
$$
z = a + b \ci \longleftrightarrow z = a {\bf1} + b \qi.
$$
In particular, we will often use a copy of the complex plane in the span of $\qj, \qk$ written as $\C \qj$. For $z \in \C$
$$
z \qj = (a + b \ci ) \qj = a \qj + b \qk.
$$
Note that $(a + b \ci) \qj = z \qj = \qj \bar z = \qj (a - b \ci)$.

For clarity we use $\Imc : \C \to \R$ and $\Imq: \H \to {\rm Im}\ {\H}$ to distinguish between the complex and quaternionic imaginary part. Note, in particular, that under these identifications, for $z = (a + b \ci) \in \C$, $\Imc z = b$ while $\Imq  = b \ \qi$. There is no ambiguity for the real part $\Re z = a$.

We will extensively use the actions of quaternions on $\R^3$. For $q \in \H$ and $X \in \R^3$, the action
$X \mapsto q^{-1} X q$ 
rotates $X$ about an axis parallel to $\Imq q$,
while $X \mapsto \overline{q} X q$ rotates about $\Imq q$ and scales by $|q|^2$.


\section{Isothermic surfaces with one family of planar curvature lines: local theory}
\label{sec:isothermic-planar-u-local}

Darboux proved that isothermic surfaces with one family of planar curvature lines come in three types. They are surfaces of revolution, have their curvature line planes tangent to a cylinder, or have their curvature line planes tangent to a cone~\cite{Darboux1883}. The cone case is \emph{generic} in the sense that the normals of the curvature line planes span three dimensional space. We call the corresponding surfaces \emph{isothermic surfaces with one generic family of planar curvature lines}. Darboux focuses on the generic case. He derives explicit formulas in terms of theta functions for the family of planar curves, together with their corresponding cone tangent lines~\cite{Darboux1883}. Remarkably, he further integrated the structure equations to find the immersion coordinates.

Complex analytic arguments form the basis of Darboux's calculations. As written, the formulas suggest a real reduction where the involved theta functions are defined on a rectangular lattice. This choice, however, leads to isothermic surfaces with one family of planar curvature lines that are not closed, implying that such surfaces cannot be tori.

As it turns out, there is a second real reduction on rhombic lattices. We show in Section~\ref{sec:isothermic-planar-u-global} these yield isothermic surfaces with one family of planar curvature lines that are closed, which can be further closed into tori.

In this section we extend Darboux's local classification by deriving the second real rhombic reduction throughout. We follow Darboux's method to integrate the structure equations for the metric. Then, starting from Section~\ref{subsubsec:conformalFrame} we take a different approach. The analytic extension of the metric gives a family of plane curves, and we use quaternions to rotate these curves into the surface's planar curvature lines.

Moreover, in Section~\ref{sec:hyperbolic-elastica} we show that the each plane curve is an area constrained elastica in the hyperbolic plane. 

\subsection{Structure equations}
The Gauss--Codazzi equations for an isothermic surface $f(u,v)$ in terms of the metric and principal curvatures are \eqref{eq:Gauss_isothermic} and \eqref{eq:Codazzi_isothermic}. Instead of the principal curvatures, Darboux \cite{Darboux1883} introduces the functions
\begin{equation}
	\label{eq:PandQ}
	\bP=k_1e^h \quad \text{and} \quad \bQ=k_2 e^h,
\end{equation}
leading to the third fundamental form expressed as $\<dn,dn>=\bP^2 du^2+\bQ^2 dv^2.$

\begin{proposition}
	The integrability conditions of an isothermic surface with metric $e^{2h}$ and curvature functions $\bP,\bQ$ defined by (\ref{eq:PandQ}) are
	\begin{align}
	\label{eq:mPQGC}
	h_{uu} + h_{vv} = - \bP \bQ, \\
	\label{eq:PvGC}
	\bP_v = h_v \bQ, \\
	\label{eq:QuGC}
	\bQ_u = h_u \bP.
	\end{align}
\end{proposition}

Since $f$ is in curvature line parametrization, a $u$-curve of $f$ is planar exactly when its Gauss map $n$ traces a circular arc on the unit sphere. The geodesic curvature of a $u$-curve of $n$ is $\frac{\bP_v}{\bP \bQ}$, and circular arcs are characterized by having constant geodesic curvature. Thus, the $u$-curves of $f$ are planar if and only if 
\begin{align}
	\label{eq:zeroUDerivativeGeodesicCurvature}
\frac{\partial}{\partial u}\left(\frac{\bP_v}{\bP \bQ}\right) = 0.
\end{align}

Combining these four equations, Darboux found an equivalent system that exhibits a more recognizable structure for the metric. The planarity condition has an integration freedom that only depends on $v$, which is incorporated into a \emph{reparametrization function} $v \mapsto \w(v)$.

\begin{theorem}[Darboux~\cite{Darboux1883}]
	\label{lem:isothermicPlanarW}
	An isothermic surface $f(u,v)$ with one family ($u$-curves) of planar curvature lines is characterized by the following equations.
	\begin{enumerate}
		\item (Harmonic) The metric factor $h$ is harmonic with respect to $u + \ci \w(v)$, for some \emph{reparametrization function $v \mapsto \w(v)$}.
		\begin{align}
				\label{eq:harmonic} h_{uu}+h_{\w\w} = 0.
		\end{align}
		\item (Riccati) The metric solves a Riccati equation with real valued coefficients $U(u)$ and $U_1(u)$ that only depend on $u$.
		\begin{align}
			\label{eq:riccati} h_u = U(u) e^h + U_1(u) e^{-h}.
		\end{align}
		\item The curvature functions satisfy
		\begin{align}
		\label{eq:thirdFFPQW}
		\bP = h_{\w} \sqrt{1-\w'(v)^{2}} & \quad \text{and} \quad \bP_{\w} = h_{\w}\bQ.
		\end{align}
	\end{enumerate}
\end{theorem}
\begin{proof}
	We show how to obtain the equations in the theorem statement. The arguments can be reversed to prove equivalence.
	
	Integrate the planarity condition \eqref{eq:zeroUDerivativeGeodesicCurvature} to find 
	\begin{align}
		\bP \bQ = V(v) \bP_v,
	\end{align}
	for some real valued function $V(v)$. Using \eqref{eq:PvGC} we find 
	\begin{align}
		\label{eq:PQwithV}
		\bP = V(v) h_v \quad \text{and} \quad \bQ = V'(v) + V(v) \frac{h_{vv}}{h_v},
	\end{align}
	whose product substituted into \eqref{eq:mPQGC} gives
	\begin{align}
		\label{eq:mPQwithV}
		h_{uu} + (1 + V(v)^2) h_{vv} + V(v) V'(v) h_v = 0.
	\end{align}
	Now, define a real valued reparametrization function $\w(v)$ by
	\begin{align}
		\w'(v) = \frac{1}{\sqrt{1 + V(v)^2}}.
	\end{align}
	This change of variables turns \eqref{eq:mPQwithV} into the harmonic equation \eqref{eq:harmonic} and the left equation of \eqref{eq:PQwithV} and \eqref{eq:PvGC} into the curvature equations \eqref{eq:thirdFFPQW}.
	
	To prove the Riccati equation \eqref{eq:riccati} we use the remaining \eqref{eq:QuGC}. From \eqref{eq:PQwithV}, $\bQ_u =  V(v) \frac{\partial}{\partial u} \left( \frac{h_{vv}}{h_v} \right)$ and $h_u \bP = V(v) h_u h_v$, so
	\begin{align}
		h_u h_v = \frac{\partial}{\partial u} \left( \frac{h_{vv}}{h_v} \right).
	\end{align}
	Using the equality of mixed derivatives, this can be rewritten as
	\begin{align}
		\label{eq:huhw}
		h_u h_{\w} = \frac{\partial}{\partial \w} \left( \frac{h_{u\w}}{h_{\w}} \right).
	\end{align}
	The Riccati equation is found after two integrations with respect to $\w$. First, multiply by $\frac{h_{u\w}}{h_{\w}}$ and integrate with respect to $\w$. Solve for $h_{\w}$ in the result and integrate with respect to $\w$, again. Finally, solve for $h_u$ and observe that the coefficients of $e^{h}$ and $e^{-h}$ are free functions that only depend on $u$. Call them $U(u)$ and $U_1(u)$, respectively.
\end{proof}

\begin{remark}
	\begin{itemize}
		\item These characteristic equations assume the metric depends on both variables. The case when the metric depends on a single variable corresponds to surfaces of revolution.
		\item For global considerations, we give the precise assumptions on admissible reparametrization functions $\w(v)$ in Definition~\ref{def:admissibleReparametrizationFunction}.
	\end{itemize}
\end{remark}

\begin{corollary}
	There exists a real valued function $U_2(u)$ satisfying
	\begin{align}
		\label{eq:hwSquaredU2}
		h_{\w}^2 = - U_1(u)^2 e^{-2h} + 2 U_1'(u) e^{-h} - U_2(u) - 2 U'(u)e^h - U(u)^2 e^{2h}. 
	\end{align}
\end{corollary}
\begin{proof}
	Differentiate the Riccati equation \eqref{eq:riccati} with respect to $u$ and substitute $h_{uu}$ into the harmonic equation \eqref{eq:harmonic}. Now, multiply by $2h_{\w}$ and integrate with respect to $\w$ to get \eqref{eq:hwSquaredU2}.
\end{proof}

\subsection{Coefficients of the Lam\'e and Riccati equations}

Darboux derived complex analytic formulas for $U(u), U_1(u)$, and $U_2(u)$ by studying the compatibility between the harmonic and Riccati equations. These formulas are expressed in terms of one-dimensional theta functions $\vartheta_i$, defined on a lattice spanned by $\pi$ and $\tau \pi$, and the nome is $q = e^{\ci \pi \tau}$. We use the convention of Whittaker and Watson \cite[Sec. 21.11]{whittaker_watson_1996}. 
In particular,
\begin{align}
\vartheta_4(z, q) = 1 + 2 \sum_{n=1}^\infty (-1)^n q^{n^2} \cos 2nz
\end{align}
is $\pi$-periodic and $\pi\tau$-quasiperiodic as a function of $z$
\begin{align}
\vartheta_4(z + \pi,q) = \vartheta_4(z,q), \quad \vartheta_4(z + \pi \tau, q) = - q^{-1} e^{-2 \ci \z} \vartheta(z, q).
\end{align}
The other $\vartheta_i$ are
\begin{align}
\vartheta_1(z,q)=-\ci q^\frac14 e^{\ci z} \vartheta_4&(z + \frac{\pi \tau}{2},q)\\ \vartheta_2(z,q)=q^\frac14 e^{\ci z} \vartheta_4(z+\frac{\pi+\pi \tau}{2},q),\quad &
\vartheta_3(z,q)=\vartheta_4(z + \frac\pi2,q).\end{align}
Moreover, $\vartheta_1(-z,q) = -\vartheta_1(z,q)$ is an odd function of $z$ while the others are even $\vartheta_i(z,q)=\vartheta_i(-z,q)$ for $i = 2,3,4$.
\begin{lemma}
	The real valued functions $U(u)$ and $U_1(u)$ are real reductions of the two linearly independent solutions of the integer Lam\'e equation
	\begin{align}
		\label{eq:complexUU1LameC1}
		\frac{U''}{U} &= \frac{U_1''}{U_1} = C_1 - 8UU_1 \text{ for some constant } C_1 \in \C,
	\end{align}
	and are given explicitly by
	\begin{align}
		\label{eq:complexUWithRho}
		U(z) &= -\ci \rho \frac{\vartheta_1'(0 , q)}{2 \vartheta_4(\omega, q)} \frac{\vartheta_1(z + \omega, q)}{\vartheta_4(z, q)} e^{-z \frac{\vartheta_4'(\omega, q)}{\vartheta_4(\omega, q)}} \quad \text{ and } \quad \\
		U_1(z) &= -\ci \frac{1}{\rho}\frac{\vartheta_1'(0, q)}{2 \vartheta_4(\omega, q)} \frac{\vartheta_1(z - \omega, q)}{\vartheta_4(z, q)} e^{z \frac{\vartheta_4'(\omega, q)}{\vartheta_4(\omega, q)}}.
		\label{eq:complexU1WithRho}
	\end{align}
	The product $U(z)U_1(z)$ is an elliptic function with a double pole at $z=\frac{\pi\tau}{2}$.
	The function $U_2(u)$ is a real reduction of
	\begin{align}
		\label{eq:complexU2}
		U_2(z) = C_1 - 6 U(z)U_1(z).
	\end{align}
\end{lemma}
\begin{proof}
	Differentiating \eqref{eq:hwSquaredU2} with respect to $u$ and inserting $h_{u\w}$ found by differentiating \eqref{eq:riccati} with respect to $\w$ gives an equation of the form $A_{-1}(u)e^{-h(u,\w)} + A_0(u) + A_1(u)e^{h(u,\w)} = 0$. This equation must hold for all $\w$, and the coefficients only depend on $u$. Therefore each coefficient must vanish identically, implying
	\begin{align}
	\label{eq:complexUSystem}
		\frac{U_1''}{U_1} = \frac{U''}{U} = U_2-2 U_1 U  \text{ and }
		(U_2 + 6 U U_1)' = 0.
	\end{align}
	
	Integrating the last equation proves \eqref{eq:complexU2}, which then implies \eqref{eq:complexUU1LameC1}.
	Note that the Wronskian $U'U_1 - U U_1'$ is a constant $C$. Defining the product $\theta = U U_1$ one verifies that 
	\begin{align}
	\theta''' = 4 \theta' (C_1 - 12 \theta).		
	\end{align}
Integrating once
$$
\theta''=4C_1\theta -24 \theta^2+C_2,
$$
then multiplying by $\theta'$ and integrating again, we obtain
\begin{equation}
\label{eq:elliptic_for_theta}
\theta'^2=-16\theta^3+4C_1\theta^2+2C_2\theta +\tilde{C}.
\end{equation}
Note that $\tilde{C}=C^2$. Indeed, the identity $\theta'^2-C^2=4\theta U' U_1' $	
implies that for $\theta=0$ one has $\theta'^2=C^2$. 

Equation \eqref{eq:elliptic_for_theta} includes three parameters $C_1, C_2, C^2$, and its general solution depends on one integration constant. Since the equation is autonomous this constant is an arbitrary shift of the variable $z\to z-z_0$.  It is known that it can be explicitly solved by elliptic functions.  Any elliptic function with a single pole of the form $\frac{1}{4(z-z_0)^2} $ solves \eqref{eq:elliptic_for_theta} with some coefficients $C_1, C_2, C^2$.
The free parameters in such a solution 
\begin{equation}
\label{eq:theta}
\theta(z, \omega)=-\frac{\vartheta_1'^2(0)}{4\vartheta_4^2(\omega)}\frac{\vartheta_1(z+\omega)\vartheta_1(z-\omega)}{\vartheta_4^2(z)}.
\end{equation}
are its zero $\omega$ and the lattice parameter $\tau$. Two additional parameters can be obtained by an arbitrary affine transformation $\alpha^2\theta(\alpha z+\beta)$ which  preserves the form of the equation. This transformation changes only the parametrization of the surface but not its geometry.
Observe that the variables in \eqref{eq:theta} can be separated
$$
\theta(z,\omega)=g(z)-g(\omega), \quad \text{where}\quad
g(s)=-\frac{\vartheta_1'^2(0)}{4\vartheta_4^2(0)}\frac{\vartheta_1^2(s)}{\vartheta_4^2(s)}.
$$

It is obvious that \eqref{eq:theta} is the product of \eqref{eq:complexUWithRho} and \eqref{eq:complexU1WithRho}.
It remains to show that $U$ and $U_1$ are solutions of the Lam\'e equation 
\begin{equation}
\label{eq:lame}
y''=(C_1 -8\theta) y
\end{equation}
with the potential \eqref{eq:theta}.  
This is the one-gap potential in the finite-gap integration theory (see, for example \cite{BBEIM}), and the simplest case of the Lam\'e equation. Its solutions were obtained by Hermite. We will show that \eqref{eq:complexUWithRho} is a solution of \eqref{eq:lame} giving a proof typical for the finite gap integration theory.
The claim for $U_1$ follows by the replacement $z\to -z$.
The function $U(z,\omega)$ given by \eqref{eq:complexUWithRho} is an elliptic function  of $\omega$ with the same period lattice. Its only singularity is at $\omega=\frac{\pi\tau}{2}$ and is of the form
\begin{equation}
\label{eq:U_asymptotics}
U(z,\omega=\frac{\pi\tau}{2}+p)=-\frac{\ci\rho}{2p} \left(1+ (\xi(z)-cz)p +O(p^2) \right) e^{-\frac{z}{p}},\quad p\to 0,
\end{equation}
with 
$$
\xi(z)=\frac{\vartheta'_4(z)}{\vartheta_4(z)},\quad c=\frac{\vartheta_1'''(0)}{3 \vartheta_1'(0)}.
$$
To derive \eqref{eq:U_asymptotics} we expand
\begin{eqnarray*}
\frac{\vartheta_4'(\frac{\pi\tau}{2}+p)}{\vartheta_4(\frac{\pi\tau}{2}+p)}=
\frac{\vartheta_1'(p)}{\vartheta_1(p)}-\ci=\frac{1}{p}-\ci + cp+ O(p^2),\\
\frac{\vartheta_1(\frac{\pi\tau}{2}+p+z)}{\vartheta_4(\frac{\pi\tau}{2}+p)}=
\frac{\vartheta_4(p+z)}{\vartheta_1(p)}e^{-\ci z}=
\frac{1}{p}\frac{\vartheta_4(z)}{\vartheta_1'(0)}e^{-\ci z}\left(1+\xi(z)p+O(p^2) \right).
\end{eqnarray*}
Together with
$$
g(\frac{\pi\tau}{2}+p)=\frac{1}{p^2}+\alpha+O(p^2), \quad \alpha= \frac{\vartheta_4''(0)}{\vartheta_4(0)}-c
$$
this implies
\begin{equation}
\label{eq:asymptotic_lame}
\frac{\partial^2 U(z,\omega)/\partial z^2}{U(z,\omega)}-g(\omega)+\alpha+2\xi'(z) -2c= O(p), \ \ p\to 0.
\end{equation}
The left hand side is an elliptic function of $\omega$ with at most one simple pole, therefore a constant. 
Due to \eqref{eq:asymptotic_lame} this constant must be zero.

To identify this with \eqref{eq:lame} we should check that
$$
g(\omega)-\alpha-2\xi'(z)+2c=C_1(\omega)-2g(\omega)+2g(z).
$$
The functions $g(z)-\alpha$ and $c-\xi'(z)$ coincide since they are both elliptic functions of $z$ with a double pole at $z=\frac{\pi\tau}{2}$ of the form
$1/(z-\frac{\pi\tau}{2})^2+o(1)$. 
Finally we obtain
$$
C_1=3g(\omega)-3\alpha.
$$
\end{proof}

The preceding formulas are valid for all complex values of $z \in \C$, parameters $\rho, \omega \in \C$, and on every lattice with complex $\tau \pi \in \C$. To integrate a real isothermic surface with coordinate $u \in \R$ and parameter $\omega \in \R$, it is necessary that $U(u)$ and $U_1(u)$ are real valued. Real reductions occur when the lattice on which the theta functions are defined is symmetric with respect to complex conjugation. This happens in exactly two cases: when the lattice is either rectangular $\tau \in \ci \R$, or rhombic $\tau \in \frac12 + \ci \R$.

\begin{remark}
	\label{rem:realReductionsShiftToRhombic}
	Darboux derived complex formulas, but focused on the rectangular case with $\rho = - \ci$. Note that the reality conditions on $U(u)$ and $U_1(u)$ only determine $\rho$ up to a real factor.
	
	The rhombic case is recovered from the complex formulas by shifting $u, \omega$ and choosing $\rho$ as follows:
	\begin{align}
		\label{eq:shiftToRhombic}
		z \mapsto u + \left( \frac{\pi}{2} + \tau \frac{\pi}{2} \right), \,
		\omega \mapsto \omega - \left( \frac{\pi}{2} + \tau \frac{\pi}{2} \right),
		\rho = e^{i \omega} e^{\left( \frac{\pi}{2} + \tau \frac{\pi}{2} \right) \frac{\vartheta_2'(\omega , q)}{\vartheta_2(\omega , q)}}.
	\end{align}
\end{remark}

\begin{proposition} The functions $U(u)$ and $U_1(u)$ are real valued for all $u \in \R$ and parameter $\omega \in \R$ if and only if one of the following two cases holds.
\label{prop:UU1BothCases}
	\begin{enumerate}
		\item(Rectangular) The lattice is rectangular, i.e., $\tau \in \ci \R$. $U(u),U_1(u)$ are given by \eqref{eq:complexUWithRho}, \eqref{eq:complexU1WithRho}, respectively, with  $\rho = - \ci$ and $z=u$.
		\item(Rhombic) The lattice is rhombic, i.e., $\tau \in \frac12 + \ci \R$, and
		\begin{align}
		\label{eq:rhombicU}
		U(u) &= - \frac{\vartheta_1'(0 , q)}{2 \vartheta_2(\omega, q)} \frac{\vartheta_1(u + \omega, q)}{\vartheta_2(u, q)} e^{-u \frac{\vartheta_2'(\omega, q)}{\vartheta_2(\omega, q)}}, \\
		U_1(u) &= + \frac{\vartheta_1'(0, q)}{2 \vartheta_2(\omega, q)} \frac{\vartheta_1(u - \omega, q)}{\vartheta_2(u, q)} e^{u \frac{\vartheta_2'(\omega, q)}{\vartheta_2(\omega, q)}}.
		\label{eq:rhombicU1}
		\end{align}
	The nome is $q = e^{\ci \pi \tau}$.
	\end{enumerate}
\end{proposition}
\begin{proof}
	Consider the complex formulas \eqref{eq:complexUWithRho} and \eqref{eq:complexUWithRho} for $U(u)$ and $U_1(u)$, respectively. The two real reductions are given by a rectangular lattice or a rhombic lattice. Darboux studied the rectangular case, we prove the rhombic case.
	
	On a rhombic lattice $\tau \in \frac12 + \ci \R$, the shifts and choice of $\rho$ in \eqref{eq:shiftToRhombic} lead to \eqref{eq:rhombicU} and \eqref{eq:rhombicU1}. To verify this, use
$
		\vartheta_4\left(z - \left( \frac{\pi}{2} + \tau \frac{\pi}{2} \right)\right) = e^{\ci \left( z - \frac\pi4 \tau\right)} \vartheta_2(z)$ and $
		\vartheta_1\left( z + (\pi + \tau \pi) \right) = e^{- \ci (2 z + \tau \pi)  }\vartheta_1(z)$, so $
		\frac{\vartheta'_4\left(z - \left( \frac{\pi}{2} + \tau \frac{\pi}{2} \right)\right)}{\vartheta_4\left(z - \left( \frac{\pi}{2} + \tau \frac{\pi}{2} \right)\right)} = \ci + \frac{\vartheta'_2(z)}{\vartheta_2(z)}.$
	Moreover, on a rhombic lattice with $\tau = \frac12 + \ci s$ for some $s \in \R$, the nome is purely imaginary: $q = e^{\ci \pi \tau} = \ci e^{-\pi s} \in \ci \R$.
	We therefore find
	\begin{equation}
	\label{eq:theta_conjugation}
		\overline{\vartheta_1(z,q)} = e^{-\ci\frac{\pi}{4}} \vartheta_1(\overline z,q),\quad
		\overline{\vartheta_2(z,q)} = e^{-\ci\frac{\pi}{4}} \vartheta_2(\overline z,q).
	\end{equation}
This implies the real valuedness of $U(u)$ and $U_1(u)$ exactly when $u$ and $\omega$ are both real.
\end{proof}

The coefficients $U(u)$ and $U_1(u)$ of the Riccati equation \eqref{eq:riccati} are now known.

\subsection{Computing the metric}

Darboux achieved a remarkable feat. He derived an explicit formula for $e^h$ by simultaneously integrating the Riccati and harmonic equations. By applying some rather nontrivial arguments and transformations Darboux has shown that the general solution of the complexified Ricatti equation  
\begin{equation}
\label{eq:Ricatti_complex}
h_z = U(z) e^h + U_1(z)e^{-h}
\end{equation}
with the coefficients \eqref{eq:complexUWithRho}, \eqref{eq:complexU1WithRho} is  given by
\begin{equation}
\label{eq:h_complex}
e^{h(z)}=	 - \ci \frac{1}{\rho} \frac{\vartheta_4\left(\frac{z  - \omega}{2}+ c\right)\vartheta_4\left(\frac{z  - \omega}{2}- c\right)}{\vartheta_1\left(\frac{z + \omega}{2}+ c\right)\vartheta_1\left(\frac{z  + \omega}{2}- c\right)} e^{z \frac{\vartheta_4'(\omega)}{\vartheta_4(\omega)}},
\end{equation}
where $c\in \C$ is an an arbitrary complex constant. 

A verification of this fact leads to a rather involved identity for theta functions. Substituting \eqref{eq:h_complex} into \eqref{eq:Ricatti_complex} we obtain
\begin{multline}
\label{eq:theta_identity_metric}
\frac{\vartheta_1'(0)}{\vartheta_4(z)\vartheta_4(\omega)}
\left(
\frac{\vartheta_4(\frac{z-\omega}{2}+c)\vartheta_4(\frac{z-\omega}{2}-c)}{\vartheta_1(\frac{z+\omega}{2}+c)\vartheta_1(\frac{z+\omega}{2}-c)}\vartheta_1(z+\omega)-
\right. \\ \left. 
\frac{\vartheta_1(\frac{z+\omega}{2}+c)\vartheta_1(\frac{z+\omega}{2}-c)}{\vartheta_4(\frac{z-\omega}{2}+c)\vartheta_4(\frac{z-\omega}{2}-c)}\vartheta_1(z-\omega)
\right)+
\frac{\vartheta_4'(\frac{z-\omega}{2}+c)}{\vartheta_4(\frac{z-\omega}{2}+c)}+ \\
\frac{\vartheta_4'(\frac{z-\omega}{2}-c)}{\vartheta_4(\frac{z-\omega}{2}-c)} - 
\frac{\vartheta_1'(\frac{z+\omega}{2}+c)}{\vartheta_1(\frac{z+\omega}{2}+c)}-
\frac{\vartheta_1'(\frac{z+\omega}{2}-c)}{\vartheta_1(\frac{z+\omega}{2}-c)} +
2 \frac{\vartheta_4'(\omega)}{\vartheta_4(\omega)}=0.
\end{multline}
To show that \eqref{eq:theta_identity_metric} is satisfied first observe that the left hand side, which we denote by $\phi(z,\omega,c)$, is an elliptic function of all its arguments. Considering it as a function of $c$ we check that the residues at all its possible poles $c= \pm\frac{z+\omega}{2}, \frac{\pi\tau}{2}\pm \frac{z-\omega}{2}$ vanish. Therefore the function is independent of $c$. Introducing $\tilde{\phi}(z,\omega):=\phi(z,\omega,0)$ we consider it as an elliptic function of $\omega$. The residues at all possible poles $\omega=\frac{\pi\tau}{2}, \frac{3\pi\tau}{2}, -z, z+\pi\tau$ vanish, therefore it is independent of $\omega$. On the other hand it is an odd function $\tilde{\phi}(z,-\omega)=-\tilde{\phi}(z,\omega)$, thus it vanishes identically.

Further the constant $c$ in \eqref{eq:h_complex} is independent of $z$, but is a function of $\w$. $h(z,c)$ satisfies the complexified harmonic equation 
$h_{zz} + h_{\w\w}=0$ if and only if $c=\frac{\ci\w}{2}$, up to trivial normalizations. Finally this implies
\begin{align}
	\label{eq:complexMetricWithRho}
	e^{h(z,\w)} =  - \ci \frac{1}{\rho} \frac{\vartheta_4\left(\frac{z + \ci \w - \omega}{2}\right)\vartheta_4\left(\frac{z - \ci \w - \omega}{2}\right)}{\vartheta_1\left(\frac{z + \ci \w + \omega}{2}\right)\vartheta_1\left(\frac{z - \ci \w + \omega}{2}\right)} e^{z \frac{\vartheta_4'(\omega)}{\vartheta_4(\omega)}}.
\end{align}

This complex valued equation exhibits the two real reductions, as described above in Remark~\ref{rem:realReductionsShiftToRhombic}.

\begin{proposition}The metric $e^h$ is real valued and positive for all $u \in \R$ and $\omega \in \R$ if and only if one of the following two cases holds.
	\begin{enumerate}
	\item(Rectangular) The lattice is rectangular, i.e., $\tau \in \ci \R$. 
	$e^h$ is given by \eqref{eq:complexMetricWithRho} with  $\rho = - \ci$ and $z=u$.
	\item(Rhombic) The lattice is rhombic, i.e., $\tau \in \frac12 + \ci \R$, 
	\begin{align}
		\label{eq:rhombicMetric}
		e^{h(u,\w)} &= \frac{\vartheta_2\left(\frac{u + \ci \w - \omega}{2}\right)\vartheta_2\left(\frac{u - \ci \w - \omega}{2}\right)}{\vartheta_1\left(\frac{u + \ci \w + \omega}{2}\right)\vartheta_1\left(\frac{u - \ci \w + \omega}{2}\right)} e^{u \frac{\vartheta_2'(\omega)}{\vartheta_2(\omega)}} \\
							&= \frac{\vartheta_2\left(\frac{u-\omega }{2}\right)^2}{\vartheta _2\left(\frac{u+\omega }{2}\right)^2}\frac{\left(1-\frac{\vartheta_1\left(\frac{\ci \w}{2}\right)^2 \vartheta _1\left(\frac{u-\omega }{2}\right)^2}{\vartheta _2\left(\frac{\ci \w}{2}\right)^2 \vartheta_2\left(\frac{u-\omega }{2}\right)^2}\right)}{\left(\frac{\vartheta _1\left(\frac{u+\omega}{2}\right)^2}{\vartheta _2\left(\frac{u+\omega}{2}\right)^2}-\frac{\vartheta _1\left(\frac{\ci \w}{2}\right)^2}{\vartheta _2\left(\frac{\ci \w}{2}\right)^2}\right)} e^{u \frac{\vartheta _2^{\prime }(\omega )}{\vartheta _2(\omega )}}
							\label{eq:rhombicMetricFactorized}
	\end{align}
	The nome is $q = e^{\ci \pi \tau}$.
\end{enumerate}
\end{proposition}
\begin{proof}
	In Remark~\ref{rem:realReductionsShiftToRhombic} we give the shifts and choice of $\rho$ to derive the rhombic reduction for $U(u), U_1(u)$. The same shifts and $\rho$ in \eqref{eq:complexMetricWithRho} lead to the real valued rhombic metric \eqref{eq:rhombicMetric}. It is positive due to \eqref{eq:theta_conjugation}. The factorization in the rhombic case \eqref{eq:rhombicMetricFactorized} follows from the theta function addition formulas
	\begin{align}
		\vartheta_2(0)^2 \vartheta_1(x + y)\vartheta_1(x-y) &= \vartheta_1(x)^2\vartheta_2(y)^2-\vartheta_2(x)^2
		\vartheta_1(y)^2, \\
		\vartheta_2(0)^2 \vartheta_2(x + y)\vartheta_2(x-y) &= \vartheta_2(x)^2\vartheta_2(y)^2-\vartheta_1(x)^2
		\vartheta_1(y)^2.
	\end{align}
\end{proof}

Our presentation thus far has closely followed Darboux's, together with highlighting where the rhombic real reduction is found.

We now take a different approach, using quaternions, to explicitly integrate the surface. This allows to consider both real reductions simultaneously to complete Darboux's local classification, and then consider global properties.

\subsection{Deriving the conformal frame}
\label{subsubsec:conformalFrame}
We know the metric $e^{2h}$ of an isothermic surface $f(u,v)$ with one family ($u$-curves) of planar curvature lines. For each $v_0$, the curvature line $f(u,v_0)$ is determined up to Euclidean motion from a planar curve $\gamma(u,v_0)$. Moreover, a classical result known as Joachimsthal's theorem~\cite[p.140 Ex. 59.8]{Eisenhart} states that a plane or sphere cuts a surface along a curvature line if and only if the intersection angle is constant. So, the surface normal $n(u,v_0)$ makes a constant angle $\phi(v_0)$ with the plane containing $f(u,v_0)$.

It turns out the family of tangent vector fields $\gamma_u(u,\w(v))$ is the analytic extension of the metric with respect to the complex variable $u + \ci \w$. The following lemma proves this fact. Along the way, we compute the defining vector fields required to explicitly integrate the family of curves $\gamma(u,\w)$ that form the surface $f(u,v)$.

Moreover, the harmonic equation \eqref{eq:harmonic} is derived using a \emph{reparametrization function} $\w(v)$ that relates $\w$ to the surface coordinate $v$. The geometry of this function is also clarified by the following lemma.

\begin{lemma}
	\label{lem:fPartialDerivatives}
	Let $f(u,v)$ be an isothermic surface with one family ($u$-curves) of planar curvature lines and reparametrization function $\w(v)$.
		
	The family of planar curves is determined by
	\begin{align}
		\label{eq:gammaUExpHPISigma}
		\gamma_u(u,\w) = e^{h(u, \w) + \ci \sigma(u, \w)},
	\end{align}
	where $h + \ci \sigma$ is holomorphic with respect to $u + \ci \w$.
	\begin{align}
		\label{eq:hsigmaCauchyRiemann}
		h_u(u,\w) = \sigma_{\w}(u,\w) \quad \text{ and } \quad h_{\w}(u,\w) = -\sigma_u(u,\w).
	\end{align}

	Moreover, there exists a nonvanishing quaternionic frame $\Phi(v)$ that rotates each planar curve into space, forming an angle $\phi(v)$ with the Gauss map $n(u,v)$, such that
	\begin{align}
		\label{eq:fuIsothermicPlanar}
		f_u(u,v) &= e^{h(u,\w(v))} \Phi(v)^{-1} e^{\ci \sigma(u,\w(v))} \qj \Phi(v), \\
		\label{eq:fvIsothermicPlanar}
		f_v(u,v) &= e^{h(u,\w(v))} \Phi(v)^{-1} \left( -\cos\phi(v) \, \qi + \sin\phi(v) \, e^{\ci \sigma(u,\w(v))} \qk \right) \Phi(v), \\
		\label{eq:nIsothermicPlanar}
		n(u,v) &= \Phi(v)^{-1} \left( \sin\phi(v) \, \qi + \cos\phi(v) \, e^{\ci \sigma(u,\w(v))} \qk \right) \Phi(v).
	\end{align}
	The functions $\Phi(v)$ and $\phi(v)$ depend only on the variable $v$, and are determined by the following.
	\begin{align}
		\label{eq:sincosphi}
		\sin\phi(v) &= \w'(v) , \quad \cos\phi(v) = -\sqrt{1-\w'(v)^2}. \\
		\label{eq:phiPrimePhiInverseW1Formula}
		\Phi'(v)\Phi^{-1}(v) &= -\cos\phi(v) W_1(\w(v)) \qk, 
		\end{align}
		where the complex valued function $W_1(\w)$ is defined by
\begin{align}
		\sigma_{\w}(u,\w(v)) &= 2 \Re \left(W_1(\w(v)) e^{-\ci \sigma(u,\w(v))}\right).
		\label{eq:sigmawForW1}	
\end{align}
\end{lemma}
\begin{proof}
	The family of planar curves satisfies $\gamma_u = e^{h+\ci \sigma}$ for some real valued function $\sigma$. To show that $h$ and $\sigma$ satisfy the Cauchy--Riemann equations with respect to $u + \ci \w$ we study the compatibility conditions of the surface. We start by proving the formulas for the tangents and Gauss map.
	
	For each $v_0$ there exists a rotation mapping the tangent vector field of $\gamma(u,v_0)$ into its position $f(u,v_0)$ on the surface in space, forming a constant angle $\phi(v_0)$ with the Gauss map. Thus, there exists a unit quaternion valued rotation function $\Phi(v)$ and real valued angle function $\phi(v)$, both depending only on $v$. This gives $f_u(u,v) = \Phi(v)^{-1} \gamma_u(u,\w(v)) \qj \Phi(v)$, which is rewritten as \eqref{eq:fuIsothermicPlanar}. Formula \eqref{eq:nIsothermicPlanar} for the Gauss map $n$ is derived by noting that $e^{\ci \sigma} \qj$ has orthogonal field $\ci e^{\ci \sigma} \qj = e^{\ci \sigma} \qk$ in the $\qj,\qk$-plane, and the angle between $n$ and the rotated $\qj,\qk$-plane is $\phi$. In isothermic coordinates $f_v = n \times f_u = \frac12[n,f_u]$, proving \eqref{eq:fvIsothermicPlanar}.

	A direct computation shows the compatibility condition $f_{uv} = f_{vu}$ is equivalent to
	\begin{equation}
		\begin{aligned}
			\label{eq:prePhifuvCompatibility}
			h_v e^{\ci \sigma} &\qj + \sigma_v e^{\ci \sigma} \qk + [e^{i \sigma} \qj, \Phi'\Phi^{-1}] \\
			= - \sigma_u \sin\phi e^{\ci \sigma} &\qj + h_u \sin\phi e^{\ci \sigma} \qk  -h_u \cos\phi \, \qi. 
		\end{aligned}
	\end{equation}
	Since $\qi, e^{\ci \sigma} \qk, [e^{i \sigma} \qj, \Phi'\Phi^{-1}]$ are orthogonal to $e^{\ci \sigma} \qj$, equating the $e^{\ci \sigma} \qj$ components gives $h_v = - \sigma_u \sin\phi$. Rewriting this in terms of $\w(v)$ is
	\begin{align}
		\label{eq:sinPhi}
		h_{\w}\w'(v) = - \sigma_u \sin\phi.
	\end{align}
	Computing the curvature function $\bP$ gives another equation relating $h_{\w}$ and $\sigma_u$. From \eqref{eq:thirdFFPQW} we know $\bP = h_{\w} \sqrt{1 - \w'(v)^2}$. On the other hand, along a $u$ curvature line we have $n_u = - \bP e^{-h} f_u$, so differentiating \eqref{eq:nIsothermicPlanar} gives $n_u = - \sigma_u \cos \phi e^{-h} f_u$. Thus, $p = \sigma_u \cos\phi$, and so
	\begin{align}
		\label{eq:cosPhi}
		h_{\w} \sqrt{1 - \w'(v)^2} = \sigma_u \cos\phi.
	\end{align}
	In terms of $\tilde \phi(v)$ defined by $\w'(v) = \sin\tilde\phi(v)$ and $\sqrt{1-\w'(v)^2} = \cos\tilde\phi(v)$, the two equations \eqref{eq:sinPhi} and \eqref{eq:cosPhi} are equivalent to
	\begin{align}
		h_{\w} = - \sigma_u \quad \text{ and } \quad \tilde \phi = \pi - \phi.
	\end{align}
	Note the second identity proves \eqref{eq:sincosphi}.
	
	Now, define $\sigma_0(u,\w(v))$ so that $h + \ci \sigma_0$ is holomorphic with respect to $u + \ci \w$. From $h_{\w} = - \sigma_u$ we know the general solution $\sigma(u,v) = \sigma_0(u,v) + \tilde \sigma(v)$. In contrast to $\sigma$ and $\sigma_0$, the integration function $\tilde \sigma$ only depends on $v$. By regauging the rotation function $\Phi(v)$, one can assume that $\tilde \sigma(v)$ vanishes for all $v$. Indeed, the tangents and Gauss map with $\Phi$ and $\sigma = \sigma_0 + \tilde\sigma$ are equivalent to the tangents and Gauss map with $\Phi \mapsto e^{\ci \frac{\tilde \sigma}{2}} \Phi$ and $\sigma \mapsto \sigma_0$.
	
	Therefore $h + \ci \sigma$ is holomorphic with respect to $u + \ci \w$.
	
	It remains to prove the structure of $\Phi'(v) \Phi^{-1}(v)$. Using holomorphicity and that $\w'(v) = \sin\phi(v)$, we have $h_v = -\sigma_u \sin\phi$ and $\sigma_v = h_u \sin\phi$. The compatibility condition \eqref{eq:prePhifuvCompatibility} reduces to
	\begin{align}
		[e^{i \sigma} \qj, \Phi'\Phi^{-1}] = -h_u \cos\phi \, \qi. 
	\end{align}
	The right hand side of this equation lies parallel to $\qi$. Therefore, $\Phi' \Phi^{-1}$ must lie in the $\qj,\qk$-plane, and can be written in the form of \eqref{eq:phiPrimePhiInverseW1Formula} for some complex valued function $W_1(\w(v)))$ satisfying
	\begin{align}
		[e^{i \sigma(u,\w(v))} \qj, W_1(\w(v)) \qk] = h_u(u,\w(v)) \qi.
	\end{align}
	Expanding the left hand side and using $h_u = \sigma_{\w}$ yields \eqref{eq:sigmawForW1}, the equivalent complex equation for $W_1$.
\end{proof}

\begin{remark}
	The notation for the complex valued function $W_1(\w) = \overline{W(\w)}$ comes from $\sigma_{\w} = 2 \Re\left(W_1(\w) e^{- \ci \sigma}\right) = W(\w) e^{\ci \sigma} + W_1(\w) e^{- \ci \sigma}$, which resembles the previous Riccati equation $h_u = U(u) e^h + U_1(u) e^{-h}$.
\end{remark}

\subsection{Infinitesimal rotation axis in the plane of $\gamma$}
Let $X$ represent the $\qj,\qk$-plane. For each $v$, the $u$-curvature line plane is given by $X(v) = \Phi^{-1}(v) X \Phi(v)$. For a small shift $v + \Delta v$ the plane is $X(v + \Delta v) = \Phi^{-1}(v + \Delta v) X \Phi(v + \Delta v)$.

To study the infinitesimal rotation, we define $\tilde \Phi(v) = \Phi^{-1}(v) \Phi(v + \Delta v)$ and compute using little-o notation that
\begin{align*}
	\tilde \Phi(v) &= \Phi^{-1}(v) ( I + \Delta v A(v) + o(\Delta v)) \Phi(v), \\
	&= I + \Delta v \Phi^{-1}(v) A(v) \Phi(v) + o(\Delta v).
\end{align*}
The rotation axis of the rotation $X(v) \to X(v + \Delta v), X \mapsto \tilde \Phi^{-1} X \tilde \Phi(v)$ is spanned by $\Phi^{-1}(v) A(v) \Phi(v)$, where $A(v)$ is defined by $A(v) = \Phi'(v) \Phi^{-1}(v)$. Thus we have the following proposition.

\begin{proposition}
	For each $v$, consider the plane curve $u \mapsto \gamma(u, \w(v)) \qj$ and axis spanned by $W_1(\w(v)) \qk = \ci W_1(\w(v)) \qj$ lying in the $\qj,\qk$-plane. Then, the corresponding planar $u$-curvature line of the isothermic immersions $f$ is found by rotating $\gamma$ using $\Phi(v)$. The axis of the infinitesimal rotation of this plane to its neighboring one is spanned by $\Phi^{-1}(v) W_1(\w(v)) \qk \Phi(v)$.
\end{proposition}

\subsection{Immersion formulas: generic case when curvature line planes are tangent to a cone }
We say an isothermic surface has one \emph{generic} family of planar curvature lines when the curvature line planes are tangent to a cone, since their normals span three dimensional space. The special case when the curvature line planes are tangent to a cylinder occurs when the parameter $\omega$ introduced below is $0$ or $\frac\pi2$. We discuss these after the generic setting.

Below we complete Darboux's classification result by providing explicit formulas for both the rectangular and rhombic cases. The family of curvature line planes are tangent to a cone. Examples are shown in Figure~\ref{fig:isothermic-reparametrization}.

\begin{theorem}
	\label{thm:localIsothermicPlanar}
	Every isothermic surface $f(u,v)$ with one generic family ($u$-curves) of planar curvature lines has its planes tangent to a cone, and is given, up to Euclidean motion, by	
	\begin{align}
		\label{eq:fImmersionsPlanar}
		f(u,v) &= \Phi^{-1}(v) \gamma(u, \w(v)) \qj \Phi(v), \\
		\Phi'(v) \Phi^{-1}(v) &= \sqrt{1 - \w'(v)^2} W_1(\w(v)) \qk,
		\label{eq:W1ImmersionsPlanar}
	\end{align}
	with reparametrization function $\w(v)$. There are two cases, depending on the lattice of the elliptic curve.
		\begin{enumerate}
		\item(Rectangular) The lattice is rectangular, i.e., $\tau \in \ci \R$, and $\omega \neq 0, \frac\pi2 \in \R$.
		\begin{align}
			\label{eq:gammaRect}
			\gamma(u,\w) &=  -\ci \frac{2 \vartheta _4(\omega)^2}{\vartheta_1^{\prime}(0) \vartheta _1(2 \omega)} \frac{\vartheta_1\left(\frac{u+\ci \w-3 \omega}{2}\right)}{\vartheta_1\left(\frac{u+\ci\w+\omega}{2}\right)}e^{(u + \ci \w)\frac{\vartheta_4'(\omega)}{\vartheta_4(\omega)}}. \\
			\label{eq:gammaURect}
			\gamma_u(u,\w) &= -\ci \gamma_{\w}(u,\w) = -\ci \left(\frac{\vartheta_4\left(\frac{u + \ci \w - \omega}{2}\right)}{\vartheta_1\left(\frac{u + \ci \w + \omega}{2}\right)}\right)^2e^{(u + \ci \w) \frac{\vartheta_4'(\omega)}{\vartheta_4(\omega)}}. \\
			\label{eq:W1Rect}
			W_1(\w) &= \ci \frac{\vartheta _1^{\prime}(0)}{2\vartheta _4(\omega)}\frac{\vartheta _4(\omega -\ci \w)}{\vartheta_1(\ci \w)}e^{\ci \w \frac{\vartheta_4'(\omega)}{\vartheta_4(\omega)}}.
		\end{align}
		\item(Rhombic) The lattice is rhombic, i.e., $\tau \in \frac12 + \ci \R$, and $\omega \neq 0, \frac\pi2 \in \R$.
		\begin{align}
			\label{eq:gammaRhombic}
			\gamma(u,\w) &=  -\ci \frac{2 \vartheta _2(\omega)^2}{\vartheta_1^{\prime}(0) \vartheta _1(2 \omega)} \frac{\vartheta_1\left(\frac{u+\ci \w-3 \omega}{2}\right)}{\vartheta_1\left(\frac{u+\ci\w+\omega}{2}\right)}e^{(u + \ci \w)\frac{\vartheta_2'(\omega)}{\vartheta_2(\omega)}}. \\
			\label{eq:gammaURhombic}
			\gamma_u(u,\w) &= -\ci \gamma_{\w}(u,\w) = -\ci \left(\frac{\vartheta_2\left(\frac{u + \ci \w - \omega}{2}\right)}{\vartheta_1\left(\frac{u + \ci \w + \omega}{2}\right)}\right)^2e^{(u + \ci \w) \frac{\vartheta_2'(\omega)}{\vartheta_2(\omega)}}. \\
			\label{eq:W1Rhombic}
			W_1(\w) &= \ci \frac{\vartheta _1^{\prime}(0)}{2\vartheta _2(\omega)}\frac{\vartheta _2(\omega -\ci \w)}{\vartheta_1(\ci \w)}e^{\ci \w \frac{\vartheta_2'(\omega)}{\vartheta_2(\omega)}}.
		\end{align}
	\end{enumerate}
\end{theorem}
\begin{figure}
	\centering
	\includegraphics[width=\textwidth]{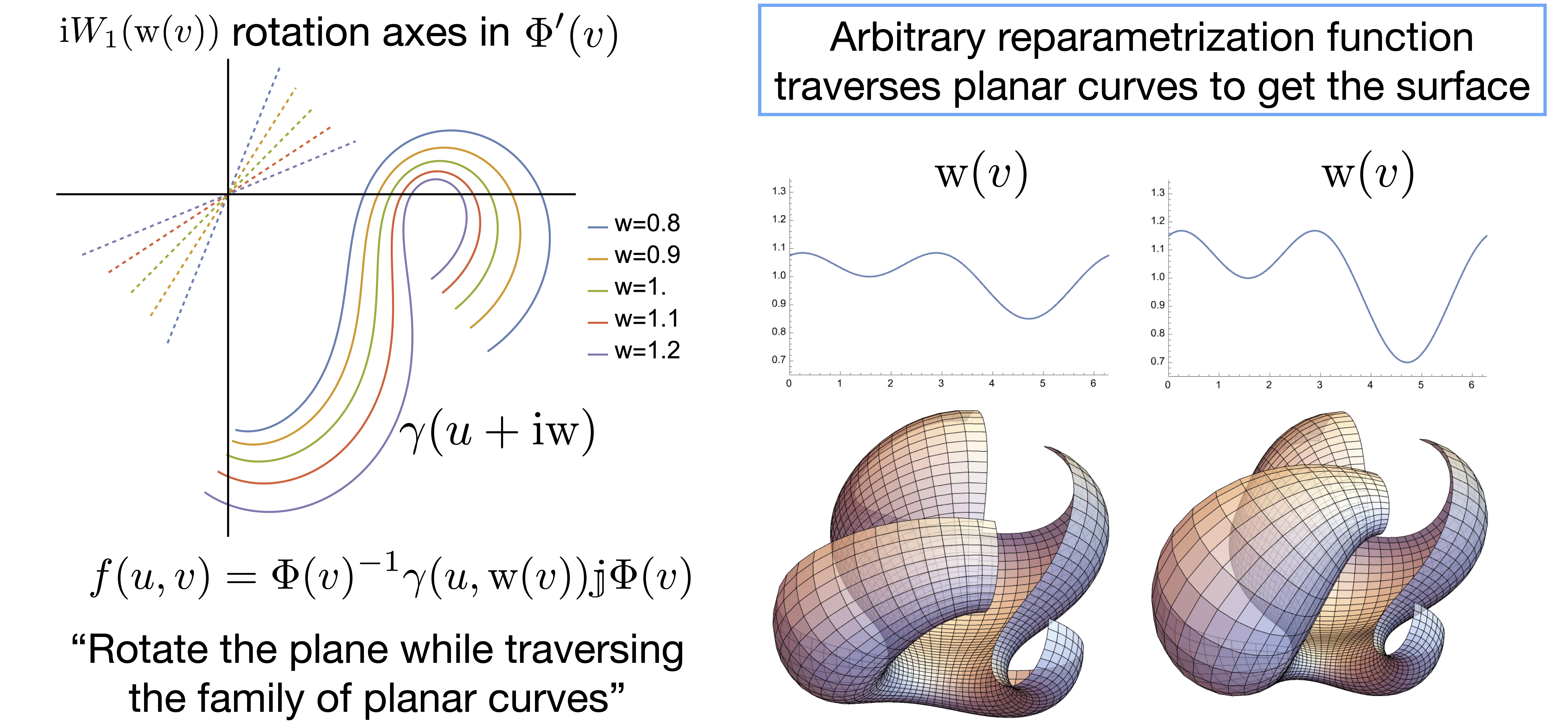}
	\caption{Left: An example family of curves $\gamma$ from a lattice of rhombic type and parameter $\omega$, together with their infinitesimal rotation axes. Choosing a reparametrization function $\w(v)$ determines an isothermic surface with one generic family of planar curvature lines. Right: Two example choices of reparametrization functions and their corresponding surfaces.}
	\label{fig:isothermic-reparametrization}
\end{figure}
\begin{proof}
	The proofs are similar for the rectangular and rhombic cases. We detail the rhombic case. We extend the metric into $\gamma_u = e^{h + \ci \sigma}$ and find $W_1(\w)$. Then we integrate the explicit formula for the family of planar curves $\gamma(u,\w)$ and find the immersion $f(u,v)$.
	
	The rhombic metric $e^h$ is given by \eqref{eq:rhombicMetric}, and is positive since
	\begin{align}
		e^{h(u,\w)} = \left( \frac{\vartheta_2\left(\frac{u + \ci \w - \omega}{2}\right)}{\vartheta_1\left(\frac{u + \ci \w + \omega}{2}\right)}e^{\frac{u}{2} \frac{\vartheta_2'(\omega)}{\vartheta_2(\omega)}} \right) \overline{\left( \frac{\vartheta_2\left(\frac{u + \ci \w - \omega}{2}\right)}{\vartheta_1\left(\frac{u + \ci \w + \omega}{2}\right)}e^{\frac{u}{2} \frac{\vartheta_2'(\omega)}{\vartheta_2(\omega)}} \right)}.
	\end{align}
	
	The analytic extension is therefore \eqref{eq:gammaURhombic}. Note we rotated by $-\ci$, when compared against Darboux's formulas for the family in the rectangular case.
	
Now we use the Riccati equation for the unitary factor in $\gamma_u = e^{h+ \ci \sigma}$, which we formulated as Proposition \ref{prop:Riccati_sigma} below, to derive the formula for  $W_1(\w)$. Indeed, $W_1(\w)$ is uniquely determined by \eqref{eq:sigmawForW1}, which coincides with \eqref{eq:Riccati_sigma}.

The formula \eqref{eq:gammaRhombic} for $\gamma$ is found by integrating \eqref{eq:gammaURhombic} with respect to the meromorphic coordinate $z = u + \ci \w$. Indeed, differentiating \eqref{eq:gammaRhombic} and equating it with \eqref{eq:gammaURhombic} is equivalent to
$$
\frac{\vartheta_1'(0)\vartheta_1(2\omega)}{\vartheta_2^2(\omega)}\frac{\vartheta_2^2(p-\frac{\omega}{2})}{\vartheta_1(p+\frac{\omega}{2})\vartheta_1(p-\frac{3\omega}{2})}=
\frac{\vartheta_1'(p-\frac{3\omega}{2})}{\vartheta_1(p-\frac{3\omega}{2})}-
\frac{\vartheta_1'(p+\frac{\omega}{2})}{\vartheta_1(p+\frac{\omega}{2})}+
2\frac{\vartheta_2'(\omega)}{\vartheta_2(\omega)}.
$$	
	Both sides of this identity are elliptic functions of $p=\frac{z}{2}$. They have singularities with coinciding residues, and a zero at $z=\omega+\pi$, which yields equality.
	
	With $\sigma_{\w} = h_u$ equation \eqref{eq:sigmawForW1} is equivalent to
	\begin{equation}
		e^{h(u,\w)} h_u(u,\w) = 2 \Re \left(W_1(\w) \overline{\gamma_u(u,\w)}\right).
	\end{equation}
As $W_1(\w)$ is independent of $u$, we integrate with respect to $u$ to find
	\begin{equation}
		\label{eq:metricAsRealPart}
		e^{h(u,\w)} = 2 \Re \left(W_1(\w) \overline{\gamma(u,\w)}\right).
	\end{equation}
To prove that the integration constant here must be zero we substitute the formulas for $e^h, \gamma$ and $W_1$:
\begin{align*}
\frac{\vartheta_2(\frac{u+\ci\w-\omega}{2})\vartheta_2(\frac{u-\ci\w-\omega}{2})}{\vartheta_1(\frac{u+\ci\w+\omega}{2})\vartheta_1(\frac{u-\ci\w+\omega}{2})}=
-\frac{\vartheta_2(\omega)}{\vartheta_1(2\omega)}
\left(
\frac{\vartheta_2(\omega-\ci\w)\vartheta_1(\frac{u-\ci\w-3\omega}{2})}{\vartheta_1(\ci\w)\vartheta_1(\frac{u-\ci\w+\omega}{2})}+
\frac{\vartheta_2(\omega+\ci\w)\vartheta_1(\frac{u+\ci\w-3\omega}{2})}{\vartheta_1(-\ci\w)\vartheta_1(\frac{u+\ci\w+\omega}{2})}
\right).
\end{align*}
Consider both parts of this identity as functions of $\tilde{u}=\frac{u}{2}$. They have the same multiplicative periodicity factors. Computing the residues of the left hand side at the poles $-\frac{\ci \w+\omega}{2}$ and $\frac{\ci \w-\omega}{2}$ we obtain the right hand side.
	
Further, to integrate the immersion $f$ we note that  
	\begin{align}
		[\gamma(u,\w) \qj, W_1(\w)\qk] = 2 \Re \left(W_1(\w) \overline{\gamma(u,\w)}\right) \qi = e^{h(u,\w)} \qi.
	\end{align}
	Therefore, we can write $f_v$ from \eqref{eq:fvIsothermicPlanar} as
	\begin{align}
		f_v &= \Phi^{-1} \left(- \cos \phi \, e^h \qi + \sin \phi \, \gamma_u \qk \right)\Phi 
			 = \Phi^{-1} \left([\gamma \qj,- \cos \phi \, W_1 \qk]
			+ \sin \phi \, \gamma_u \qk \right)\Phi \nonumber \\
			& = \Phi^{-1} \left([\gamma \qj, \Phi' \Phi^{-1}] 
			+ \gamma_v \qj \right) \Phi 
			 = (\Phi^{-1} \gamma \qj \Phi )_v. \label{eq:f_v}
	\end{align}
	 
	In the second to last equality we used that $\Phi'\Phi^{-1} = -\cos\phi W_1\qk$ from \eqref{eq:phiPrimePhiInverseW1Formula}, that $\sin\phi = \w'$ from \eqref{eq:sincosphi}, and that $\gamma_v = \w' \gamma_{\w} = \ci \w' \gamma_u$.
	
	Combinbing this expression for $f_v$ with $f_u = \Phi^{-1} \gamma_u \qj \Phi$ from \eqref{eq:fuIsothermicPlanar} proves that, up to a global constant translation, $f(u,v)$ is \eqref{eq:fImmersionsPlanar}.
\end{proof}
\begin{proposition}
\label{prop:Riccati_sigma}
The unitary factor in $\gamma_u=e^{h+\ci \sigma}$ satisfies the Riccati equation
\begin{equation}
\label{eq:Riccati_sigma}
\sigma_{\w}=We^{\ci \sigma}+W_1e^{-\ci\sigma}
\end{equation}
with functions $W_1$ and $W$ depending on $\w$ only and given by \eqref{eq:W1Rect}, \eqref{eq:W1Rhombic} and $W(\w)=\overline{W_1(\w)}$. Explicit formulas for $e^{\ci\sigma}$ for the two real types of the elliptic curve are as follows: 
		\begin{enumerate}
		\item(Rectangular)
		\begin{align}
			\label{eq:sigmaRect}
			e^{\ci\sigma(u,\w)} &=  -\ci \frac{\vartheta_4\left(\frac{u+\ci \w-\omega}{2}\right) \vartheta_1\left(\frac{u-\ci\w+\omega}{2}\right)}{\vartheta_1\left(\frac{u+\ci\w+\omega}{2}\right)\vartheta_4\left(\frac{u-\ci\w-\omega}{2}\right)}
			e^{ \ci \w\frac{\vartheta_4'(\omega)}{\vartheta_4(\omega)}},
		\end{align}
		\item(Rhombic) 
		\begin{align}
			\label{eq:sigmaRhombic}
						e^{\ci\sigma(u,\w)} &=  -\ci \frac{\vartheta_2\left(\frac{u+\ci \w-\omega}{2}\right) \vartheta_1\left(\frac{u-\ci\w+\omega}{2}\right)}{\vartheta_1\left(\frac{u+\ci\w+\omega}{2}\right)\vartheta_2\left(\frac{u-\ci\w-\omega}{2}\right)}
			e^{ \ci \w\frac{\vartheta_2'(\omega)}{\vartheta_2(\omega)}}.
		\end{align}
	\end{enumerate}
\end{proposition}
\begin{proof}
One can observe that the theta functional identities implied by equations \eqref{eq:riccati} and  \eqref{eq:Riccati_sigma} coincide up to change of variables. In particular, in the rhombic case, substituting expressions \eqref{eq:rhombicU}, \eqref{eq:rhombicU1}, \eqref{eq:rhombicMetric} into \eqref{eq:riccati} and interchanging $u\leftrightarrow \ci\w$ we obtain the identity \eqref{eq:Riccati_sigma} with the coefficients given by \eqref{eq:W1Rhombic} and \eqref{eq:sigmaRhombic}. 
\end{proof}

\subsection{Immersion formulas: special case when curvature line planes are tangent to a cylinder}
The immersion formulas require special treatment when the parameter $\omega$ is $0$ or $\frac\pi2$.
 In these cases the planes of planar curvature lines become tangent not to a cone but to a cylinder. By direct computation one can prove the following.

\begin{lemma}
\label{lem:asymptotics_omega=0}
The theta functional formulas \eqref{eq:gammaRhombic}, \eqref{eq:W1Rhombic} for $\gamma$ and $W=\bar{W}_1$  in the rhombic case have the following asymptotic behavior for $\omega\to 0$: 
\begin{eqnarray}
\gamma(u,\w)&=&\frac{c}{\omega}+\hat{\gamma}(u,\w)+ o(1),\\
W(\w)&=&\hat{W}(\w)(1+\omega d(\w)+o(\omega)),\\
c&=&-\ci \frac{\vartheta_2^2(0)}{\vartheta_1^{'2}(0)},\\
\hat{\gamma}(u,\w)&=& -\ci \frac{\vartheta_2^{''}(0)\vartheta_2 (0)}{\vartheta_1^{'2}(0)}(u+\ci \w)+
2\ci \frac{\vartheta_2^2 (0) \vartheta_1' (\frac{1}{2}(u+\ci\w))}{\vartheta_1^{'2} (0)  \vartheta_1(\frac{1}{2}(u+\ci\w))},
\label{eq:hat_gamma}\\
\hat{W}(\w)&=& \ci \frac{\vartheta_1'(0) \vartheta_2(\ci\w)}{2\vartheta_2(0) \vartheta_1(\ci\w)},
\label{eq:hat_W}\\
d(\w)&=&\frac{\vartheta_2'(\ci\w)}{\vartheta_2(\ci\w)}-\ci\w \frac{\vartheta_2''(0)}{\vartheta_2(0)}.
\end{eqnarray}
In the rectangular case $\tau\in\ci{\mathbb R}$ one should replace all theta functions $\vartheta_2$ in these formulas by $\vartheta_4$ and $\vartheta_3$ for the asymptotics $\omega\to 0$ and $\omega\to\frac{\pi}{2}$ respectively. There is no limit $\omega\to\frac{\pi}{2}$ in the rhombic case.
\end{lemma}
Observe that $\hat{\gamma}$ is complex valued, $\hat{W}$ is real valued, and $c$ and $d$ are purely imaginary.

\begin{remark}
\label{rem:thetas_234}
Since for $\tau\in\ci{\mathbb R}$ the theta constants $\vartheta_4''(0)$ and $\vartheta_3''(0)$ never vanish, there are no curves $\hat{\gamma}(u,\w)$ periodic in $u$ corresponding to the rectangular type with $\omega=0$ or $\omega=\frac{\pi}{2}$.
As explained in the next section there exists a unique  elliptic  curve of  rhombic type such that $\vartheta_2''(0|\tau_0)=0$. The corresponding curves $\hat{\gamma}(u,\w)$ are periodic in $u$.
\end{remark}
We proceed with the investigation of the rhombic case we are interested in the most. The other cases are considered in exactly the same way.

In the limit $\omega\to 0$ the formula $e^h=2 \Re(W\gamma)$ for the metric becomes 
$$
e^h=2\Re (\hat{W}\hat{\gamma})+2 cd \hat{W}
$$
with the coefficients from Lemma~\ref{lem:asymptotics_omega=0}. 
The frame satisfies 
$$
\Phi_v=\sqrt{1-\w^{' 2}}\hat{W}(\w)\qk \Phi,
$$
and can be explicitly integrated
\begin{equation}
\label{eq:a_v}
\Phi(v)=e^{a(v)\qk}\quad \text{with}\quad \frac{da}{dv}=\sqrt{1-\w^{' 2}}\hat{W}(\w).
\end{equation}
Now repeating the computations \eqref{eq:f_v} we obtain
\begin{equation}
\label{eq:f_v_omega=0}
f_v=(\Phi^{-1}\hat{\gamma}\qj\Phi)_v+ \sqrt{1-\w^{' 2}}\Phi^{-1}r\qi\Phi
\end{equation}
with $r=2cd\hat{W}$.  Integrating we come to the following result.
\begin{theorem}
\label{prop:f_omega=0}
Let $f(u,v)$ be an isothermic surface with one family ($u$-curves) of planar curvature lines parametrized as in Theorem~\ref{thm:localIsothermicPlanar} by \eqref{eq:fImmersionsPlanar} and \eqref{eq:W1ImmersionsPlanar}, but with parameter $\omega$ equal to $0$ or $\frac\pi2$. Then the planes of the planar curvature lines are tangent to a cylinder and the immersion $f(u,v)$ is given as follows.
\begin{eqnarray}
f(u,v)=\Imc(\hat{\gamma}(u,\w))\qk + \Re (\hat{\gamma}(u,\w))\qj e^{2a(v)\qk}+\nonumber\\ 
\int^v \sqrt{1-\w^{' 2}(\tilde{v})}r(w(\tilde{v}))\qi e^{2a(\tilde{v})\qk}d \tilde{v}, \label{eq:f_omega=0}
\end{eqnarray}
where in the rhombic $\omega=0$ case the coefficients are given by \eqref{eq:hat_gamma}, \eqref{eq:hat_W} and \eqref{eq:a_v}, and
$$
r(\w)=\frac{\vartheta_2(0) \vartheta_2(\ci\w)}{\vartheta_1'(0) \vartheta_1(\ci\w)}
\left(\frac{\vartheta_2'(\ci\w)}{\vartheta_2(\ci\w)}-\ci\w \frac{\vartheta_2''(0)}{\vartheta_2(0)} \right).
$$
For a rectangular lattice $\tau\in\ci{\mathbb R}$ one should replace all theta functions $\vartheta_2$ in these formulas by $\vartheta_4$ and $\vartheta_3$ for the cases $\omega= 0$ and $\omega=\frac{\pi}{2}$ respectively.
\end{theorem}
\begin{proof}
For the first term in \eqref{eq:f_v_omega=0} we have 
$$
\Phi^{-1}\hat{\gamma}\qj\Phi=e^{-a\qk}(\Re (\hat{\gamma})\qj +\Imc (\hat{\gamma})\qk)e^{a\qk}.
$$
The last term in \eqref{eq:f_v_omega=0} is a function of $v$ only and $r$ is real valued.
\end{proof}
The first two terms in \eqref{eq:f_omega=0} describe a curve in the  plane $(\qj, \qk)$  rotated about the $\qk$-axis, 
and the integral term describes its translation in the  plane $(\qi, \qj)$.

\subsection{Parametrization of the moduli space}

\begin{corollary}
	The moduli space of isothermic surfaces with one family of planar curvature lines is parametrized by a choice of real valued function $\w(v)$, a real parameter $\omega$, and a real lattice parameter $\Imc\tau$.
\end{corollary}

\section{Planar curvature lines as hyperbolic elastica} 
\label{sec:hyperbolic-elastica}
Each isothermic surface with one generic family of planar curvature lines is given explicitly by
\begin{align*}
	f(u,v) &= \Phi(v) \gamma(u, \w(v))\qj \Phi(v), \\
	\Phi'(v)\Phi^{-1}(v) &= \sqrt{1-\w'(v)^2} W_1(\w(v)) \qk.
\end{align*}
The reparametrization function $\w(v)$ traverses a family of plane curves $\gamma(u,\w)$, while $\Phi(v)$ rotates the planes about axes encoded by $\ci W_1(\w)$. 

Each plane curve is infinitesimally rotated about an axis lying in its plane. We interpret this axis as the infinite boundary of a hyperbolic plane in the conformal half plane model. This allows to prove that each curve is a critical point of total squared hyperbolic curvature subject to constraints on length, area, and quasiperiodicity with translational period. In short, each curve is a constrained hyperbolic elastica.

\subsection{Conformal model, hyperbolic curvature, and elastica}
We study curves immersed into the hyperbolic plane $\mathrm{H^2}$ using the conformal Poincar\'e half-plane model given by $\mathrm{H^2} = \C_+ = \{ z = x + \ci y \in \C \, \vert \, y > 0, x,y \in \R\}$ with metric $ds^2 = \frac{dx^2 + dy^2}{y^2}$. Since this model is conformal, each osculating circle of a hyperbolic plane curve is the same as its corresponding Euclidean curve, only measured differently. Using this observation, we compute the curvature function of a hyperbolic curve in terms of its Euclidean unit tangent vector.

\begin{lemma}
	\label{lem:hyperbolicConstantSpeed}
	Let $L, a > 0$. An immersed curve $\tilde \gamma: [0, a L] \to \C_+$ has hyperbolic parametrization with constant speed $a$ if and only if \begin{align}
		\label{eq:hyperbolicConstantSpeed}
		\Imc \tilde \gamma(u) = a |\tilde \gamma _u(u)|_{\mathrm{Euc}}.
	\end{align}
\end{lemma}
\begin{proof}
	We have $|\tilde \gamma_u(u)|^2_{\mathrm{Hyp}} = |\tilde \gamma_u(u)|^2_{\mathrm{Euc}} / (\Imc \tilde \gamma(u))^2$ for all $u \in I$. As $\Imc \tilde \gamma > 0$ we conclude $a = |\tilde \gamma_u(u)|_{\mathrm{Hyp}}$ if and only if $\Imc \tilde \gamma(u) = a |\tilde \gamma _u(u)|_{\mathrm{Euc}}$.
\end{proof}

\begin{proposition}
	Let $L, a > 0$. Consider an immersed curve $\tilde \gamma: [0, a L] \to \C_+$ with hyperbolic parametrization of constant speed $a > 0$. Let $\tilde \sigma: [0, a L] \to \R$ be the angle function describing the Euclidean unit tangents via $e^{\ci \tilde \sigma}$. Then, the hyperbolic curvature function is 
	\begin{align}
		\label{eq:hyperbolicCurvatureFormula}
		\kappa_{\mathrm{Hyp}} = \tilde \sigma _s + \cos \tilde \sigma,
	\end{align}
	where $s = u/a$ is the hyperbolic arclength parameter of $\tilde \gamma$.
\end{proposition}
\begin{proof}
The Euclidean centers $C_{\mathrm{Euc}}$ of the osculating circles are 
\begin{align}
	C_{\mathrm{Euc}} &= \tilde \gamma + \ci e^{\ci \tilde \sigma} R_{\mathrm{Euc}} \text{ with } \nonumber \\
	R_{\mathrm{Euc}} &= \kappa_{\mathrm{Euc}}^{-1} = \frac{|\tilde \gamma_u |_{\mathrm{Euc}}}{\tilde \sigma_u}, \label{eq:euclideanRadiusOfCurvature}
\end{align}
and where $R_{\mathrm{Euc}}, \kappa_{\mathrm{Euc}}$ are the Euclidean radius and curvature functions.

The hyperbolic centers $C_{\mathrm{Hyp}}$ of these osculating circles have the same real coordinate as the Euclidean centers, so we have to determine their imaginary parts. For each $u$, the pair of points $C_{\mathrm{Euc}}(u) \pm \ci R_{\mathrm{Euc}}(u)$ lie on the osculating circle and are hyperbolically equidistant to the hyperbolic center $C_{\mathrm{Hyp}}(u)$. Denote their imaginary parts by $y^\pm(u) = \Imc\left(C_{\mathrm{Euc}}(u) \pm \ci R_{\mathrm{Euc}}(u)\right)$ and $y(u) = \Imc C_{\mathrm{Hyp}}(u)$. Using \eqref{eq:euclideanRadiusOfCurvature} and \eqref{eq:hyperbolicConstantSpeed} we compute
\begin{align}
	y^\pm 
	&= \frac{|\tilde \gamma_u |_{\mathrm{Euc}}}{\tilde \sigma_u} ( a \tilde \sigma_u + \cos\tilde\sigma  \pm 1). \label{eq:yPMFormula}
\end{align}

Note that the hyperbolic distance between $\ci t^+, \ci t^- \in \C_+$ for $t^+ > t^- \in \R$ is $d_{\mathrm{Hyp}}( \ci t^+, \ci t^-) = \log \frac {t^+}{t^-}$. Thus, $y$ is hyperbolically equidistant to $y^\pm$ if and only if $\log\frac{y}{y^-} = \log\frac{y^+}{y}$, and so $y^2 = y^+ y^-$.
The hyperbolic radius of curvature $R_{\mathrm{Hyp}}$ is therefore 
\begin{align}
	R_{\mathrm{Hyp}} = \log \frac{y}{y^-} = \frac12 \log \frac{y^2}{(y^-)^2} = \frac12 \log\frac{y^+}{y^-}.
\end{align}
Plugging in $y^\pm$ from \eqref{eq:yPMFormula} yields
\begin{align}
		R_{\mathrm{Hyp}} = \frac12 \log\frac{a \tilde \sigma_u + \cos\tilde\sigma  + 1}{a \tilde \sigma_u + \cos\tilde\sigma  - 1}.
\end{align}
The hyperbolic curvature function is given by $\kappa_{\mathrm{Hyp}} = \frac{1}{\tanh R_{\mathrm{Hyp}}}$ and so we find
\begin{align}
	\kappa_{\mathrm{Hyp}} = a \tilde \sigma_u + \cos \tilde \sigma = \tilde \sigma_s + \cos \tilde \sigma,
\end{align}
where $s = u/a$ is the hyperbolic arclength of $\tilde \gamma$.
\end{proof}

The map $\kappa_{\mathrm{Hyp}}: [0, L] \to \R$ of an arclength parametrized curve determines $\tilde \gamma: [0, L] \to \mathrm{H^2} = \C_+$ uniquely up to hyperbolic motions.

Hyperbolic elastica are critical points of the elastic energy $\int_0^L \kappa_{\mathrm{Hyp}}^2 ds$ subject to the constraint of total hyperbolic length $L$. 
We want to study a particular class of quasiperiodic area constrained elastica. Namely, those that have a monodromy given by a hyperbolic translation along a geodesic. In the upper half plane model we normalize so that this translation is along the imaginary axis. In this case, hyperbolic translation is just a real scaling by $\mathrm{p} \in \R$. The period along the imaginary axis is $\mathrm{p} = \frac{\Imc \tilde \gamma(L)}{\Imc \tilde \gamma(0)}$. Thus, $\log \mathrm{p} = \log \Imc \tilde \gamma(L) - \log \Imc \tilde \gamma(0) = \int_0^L \sin \tilde \sigma \, ds$.

We constrain the enclosed area as follows. By the Gauss--Bonnet theorem the area of the hyperbolic triangle enclosed above an arclength element $ds$ is $\left( \tilde \sigma_s + \kappa_{\mathrm{Hyp}}\right) ds$. Integrating we obtain $\int_0^L \left( \tilde \sigma_s + \kappa_{\mathrm{Hyp}} \right) ds = \tilde \sigma(L) - \tilde\sigma(0) + \int_0^L \kappa_{\mathrm{Hyp}} ds.$ Quasiperiodicity means that $\sigma(L) = \tilde\sigma(0)$, so the boundary terms vanish.

\begin{definition}
An \emph{area constrained quasiperiodic hyperbolic elastica with translational monodromy} is an arclength parametrized curve $\tilde \gamma: [0, L] \to \mathrm{H^2} = C_+$ whose $\kappa_{\mathrm{Hyp}}: [0, L] \to \R$ is a critical point of the elastic energy
\begin{align}
	\label{eq:areaConstrainedElasticEnergyKHyp}
	\int_0^L \kappa_{\mathrm{Hyp}}^2 ds.
\end{align}
with fixed length $L > 0$, fixed area, and fixed translational period $\mathrm{p}$ with equal tangent directions at its endpoints.
\end{definition}

\begin{remark}
Admissible variations to derive hyperbolic elastica must preserve total hyperbolic length and boundary conditions. In the classical setup, the boundary conditions are fixed endpoints $\tilde{\gamma}(0), \tilde{\gamma}(L)$ and tangent directions $\tilde{\sigma}(0), \tilde{\sigma}(L)$. For (area constrained) quasiperiodic hyperbolic elastica only the translation period $\mathrm{p} = \frac{\Imc \tilde \gamma(L)}{\Imc \tilde \gamma(0)}$ with equal tangent directions $\tilde{\sigma}(0)=\tilde{\sigma}(L)$ are constrained. The endpoints $\tilde{\gamma}(0), \tilde{\gamma}(L)$ are not constrained.
\end{remark}

\begin{proposition}
Let $L, a > 0$. Consider an immersed curve $\tilde \gamma: [0, a L] \to \C_+$ with hyperbolic parametrization of constant speed $a > 0$. Let $\tilde \sigma: [0, a L] \to \R$ be the angle function describing the Euclidean unit tangents via $e^{\ci \tilde \sigma}$.

Then $\tilde \gamma$ is area constrained quasiperiodic hyperbolic elastica with translational monodromy if it is a critical point of the following energy $E_{\lambda, \delta}$ with Lagrange multipliers $\lambda, \delta \in \R$. 

\begin{align}
E_{\lambda,\delta}(\tilde \gamma) = \int_0^L \left( \tilde \sigma_s^2 + \cos^2 \tilde \sigma + \lambda \cos \tilde \sigma + \delta \sin \tilde \sigma \right)ds,
\end{align}
with fixed length $L > 0$ and equal tangent directions $\tilde \sigma(0) = \tilde \sigma(L)$ at its endpoints.
Here, $s = u/a$ is the hyperbolic arclength of $\tilde \gamma$.
\end{proposition}
\begin{proof}
Using \eqref{eq:hyperbolicCurvatureFormula} for $\kappa_{\mathrm{Hyp}}$ in \eqref{eq:areaConstrainedElasticEnergyKHyp}, together with Lagrange multipliers for the area and translational period constraints, we compute
\begin{align*}
	\int_0^L \kappa_{\mathrm{Hyp}}^2 + \lambda \, \kappa_{\mathrm{Hyp}} + \delta \sin \tilde \sigma \,  ds
	&= \int_0^L \left( (\tilde \sigma_s + \cos \tilde \sigma)^2 + \lambda (\tilde \sigma_s + \cos \tilde \sigma) + \delta \sin \tilde \sigma \right)ds \\
	&= \left( + 2 \sin \tilde\sigma + \lambda \tilde \sigma \right) \big \vert_0^L + E_{\lambda,\delta}(\tilde \gamma) = E_{\lambda,\delta}(\tilde \gamma).
\end{align*}
\end{proof}

The corresponding Euler--Lagrange equation describes an elliptic curve.

\begin{proposition}
	The Euler--Lagrange equation for $E_{\lambda,\delta}(\tilde \gamma)$ is expressed in terms of $\tilde Q = e^{\ci \tilde \sigma}$ by
	\begin{align}
		\tilde Q_s^2 + \frac14 \tilde Q^4 + \frac{\lambda - \ci \delta}{2} \tilde Q^3 + \mu \, \tilde Q^2  + \frac{\lambda + \ci \delta}{2} \tilde Q + \frac14 = 0,
	\end{align}
	for some constants $\lambda, \delta, \mu \in \R$.
\end{proposition}
\begin{proof}
	Variational calculus yields the Euler--Lagrange equation in the form 
	\begin{align}
		- 2 \tilde \sigma_{ss} - \underbrace{2 \cos \tilde \sigma \sin \tilde \sigma}_{\sin 2 \tilde \sigma} - \lambda \sin \tilde \sigma  + \delta \cos \tilde \sigma = 0.
	\end{align}
	Multiplying by $\tilde \sigma_s$ and integrating yields
$
		- \tilde \sigma_{s}^2 + \frac12 \cos 2 \tilde \sigma + \lambda \cos \tilde \sigma + \delta \sin \tilde \sigma + \mu = 0
$	for some $\mu \in \R$. Lastly, multiply through by $e^{2 \ci \tilde \sigma}$ and rewrite cosine and sine in terms of exponentials.
\end{proof}

\subsection{Proofs that the $\gamma$ curves are hyperbolic elastica}

\begin{figure}
	\includegraphics[width=\linewidth]{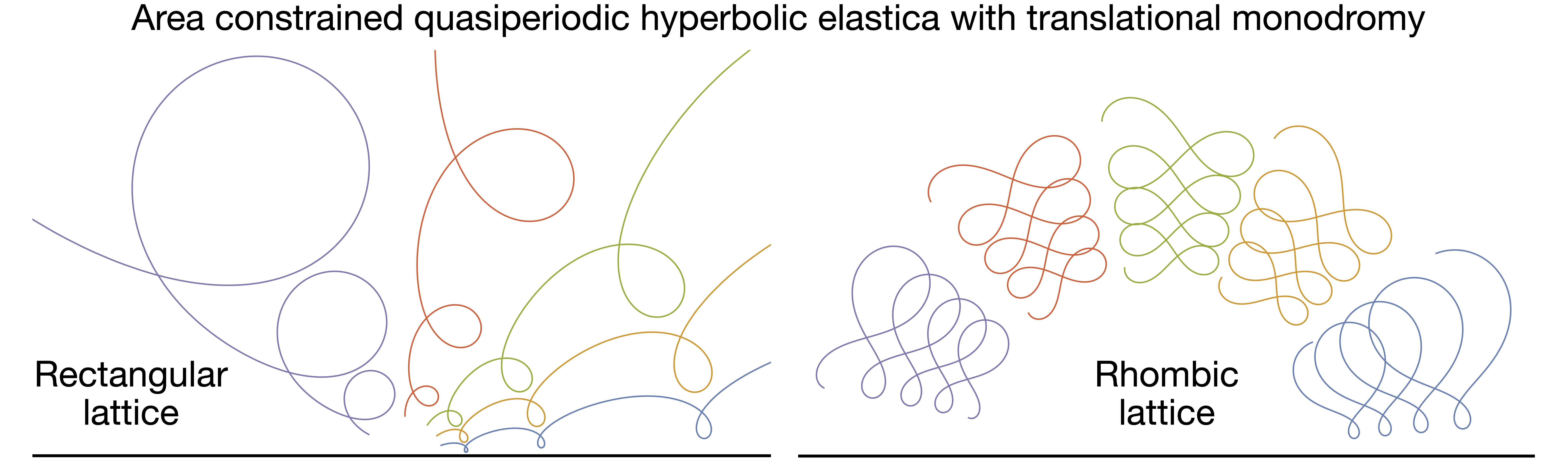}
	\caption{Example hyperbolic standardizations $u \mapsto \tilde \gamma(u,\w)$ for five values of $\w$ in the half-plane model. Left: Rectangular case with $\tau \approx 0.953571 \, \ci $ and parameter $\omega = 0.3$. Rectangular lattices never have closed curves, see Lemma~\ref{lem:rectangularNeverClosed}. Right: Rhombic case with $\tau = \frac12 + \frac{25}{78} \ci$ and parameter $\omega = 0.08$. Some rhombic lattices admit a critical value of $\omega$ so that the curves close, see Lemma~\ref{lem:rhombic_zeros}.}
	\label{fig:hyperbolicElastica-examples}
\end{figure}

Consider a planar curve $u \mapsto \gamma(u, \w)$ with its axis spanned by $\ci W_1(\w)$.
To use the standard upper half-plane model of the hyperbolic plane, we rotate so that $\ci W_1(\w)$ lies along the real axis and $\gamma$ lies above it. Examples are shown in Figure~\ref{fig:hyperbolicElastica-examples}.

\begin{definition}
For each $\w$, we define \emph{the hyperbolic standardization of $\gamma$} by
\begin{align}
\tilde \gamma(u,\w) = -\frac{|W_1(\w)|}{\ci W_1(\w)} \gamma(u, \w).
\end{align}
\end{definition}

\begin{proposition}
Let $\tilde \gamma$ be the hyperbolic standardization of $\gamma$. Then, $\tilde \gamma$ is parametrized with constant hyperbolic speed $a = 2 | W_1(\w))|.$ In particular, $s \mapsto \tilde \gamma(s/a)$ is parametrized by hyperbolic arclength.
\end{proposition}
\begin{proof}
Note that $| \tilde \gamma_u| = | \gamma_u | = e^h$. Now, rewrite \eqref{eq:metricAsRealPart} in terms of $\tilde \gamma$. This leads to $e^{h(u,\w)}$ equal to $2 \Re \left(W_1(\w) \overline{(- \ci \frac{W_1(\w)}{|W_1(\w)|} \tilde \gamma(u,\w))}\right) = 2 |W_1(\w)| \Imc \tilde \gamma.$
By Lemma~\ref{lem:hyperbolicConstantSpeed}, $\tilde \gamma$ has constant hyperbolic speed $2 |W_1(\w)|$.
\end{proof}

We use the Euclidean unit tangent vectors to study $\gamma$, and then rewrite the result in terms of $\tilde \gamma$.

\begin{lemma}
Let $e^{\ci \sigma(u, \w)}$ be the unit tangents from the family of curves $\gamma(u, \w)$. Then there exists a real valued function $W_2(\w)$ such that
\begin{align}\nonumber
	\left( \frac{\partial}{\partial u} e^{\ci \sigma(u,\w)}\right)^2 &=
	e^{4 \ci \sigma(u,\w)} W(\w)^2 - 
	2 \ci e^{3 \ci \sigma(u,\w)} W'(\w)  + \\ & +  
	e^{2 \ci \sigma(u,\w)} W_2(\w) + 
	2 \ci e^{\ci \sigma(u,\w)} W_1'(\w) + W_1(\w)^2.
	\label{eq:expISigmaEllipticCurve}
\end{align}
\end{lemma}
\begin{proof}
	By Lemma~\ref{lem:fPartialDerivatives} we have $\gamma_u = e^{h + \ci \sigma}$ with $h_u = \sigma_\w$ and $h_\w = -\sigma_u$. Thus, $h_{u\w} = h_{\w u}$ implies the harmonic equation $\sigma_{uu} + \sigma_{\w\w} = 0$. Moreover, $\sigma$ satisfies the Riccati equation \eqref{prop:Riccati_sigma} for $\sigma_\w$. Now, differentiate with respect to $\w$, replace $\sigma_\w$ using \eqref{prop:Riccati_sigma}, and plug the resulting expression for $\sigma_{\w\w}$ into the harmonic equation to find
	\begin{align}
		\sigma_{uu} + \ci e^{2 \ci \sigma} W^2 + e^{\ci \sigma} W' +e^{-\ci \sigma} W_1' - \ci e^{-2 \ci \sigma} W_1^2 = 0.
	\end{align}
Multiplying by $2 \sigma_u$ and integrating with respect to $u$ gives 
\begin{align}
		\sigma_{u}^2 + e^{2 \ci \sigma} W^2 - 2 \ci e^{\ci \sigma} W' +W_2 + 2 \ci e^{-\ci \sigma} W_1' + e^{-2 \ci \sigma} W_1^2 = 0.
\end{align}
for some real valued function $\w \mapsto W_2(\w)$. Conclude by multiplying with $e^{2 \ci \sigma}$.
\end{proof}

\begin{theorem}
	\label{thm:elasticaLowerOrder}
	Let $\tilde \gamma$ be the hyperbolic standardization of $\gamma$.  Then for each $\w$, the curve $\tilde \gamma( \cdot , \w)$ is an area constrained quasiperiodic hyperbolic elastica with translational monodromy.
	
	In particular, for each $\w$, the Euclidean unit tangents $\tilde Q = e^{\ci \tilde \sigma}$ satisfy
	\begin{align}
		\tilde Q_s^2 = \frac14 \tilde Q^4 + \bar \Lambda \tilde Q^3 + \mu \tilde Q^2 + \Lambda \tilde Q + \frac14	\end{align}
	for some $\mu \in \R$ and $\Lambda = \frac{(\frac{\partial}{\partial \w}\log W_1(\w))}{a} \in \C$. Here, $a = 2 |W_1(\w)|$ and $s = u/a$ is the hyperbolic arclength parameter.
\end{theorem}
\begin{proof}
	The Euclidean unit tangents for $\tilde \gamma$ are $e^{\ci \tilde \sigma}$. Since $\tilde \gamma$ is just a rotated version of $\gamma$, we have that $e^{\ci \tilde \sigma} = -\frac{|W_1(\w)|}{\ci W_1(\w)} e^{\ci \sigma}$. Now, rewrite \eqref{eq:expISigmaEllipticCurve} in terms of $e^{\ci \tilde \sigma}$, using that $W_1 = \overline W$, so $| W_1 | = (W W_1)^{1/2}$.
\end{proof}
\begin{remark}
The coefficients $\mu$ and $\Lambda$ are independent of $u$, but are given in terms of $W_1$~\eqref{eq:W1Rect},\eqref{eq:W1Rhombic}, which depends on $\w$, $\omega$ and lattice parameter $\Imc \tau$.
\end{remark}

\section{Isothermic surfaces with one family of closed planar curvature lines: global theory} 
\label{sec:isothermic-planar-u-global}
In this section we classify isothermic cylinders with one family of closed planar curvature lines and study their global geometric properties. We focus on the generic case when the curvature line planes are tangent to a cone. In some instances we provide formulas in the limit case when the planes are tangent to a cylinder. Surfaces of constant mean curvature with one family of closed planar curvature lines, such as Wente tori~\cite{Wente1986,Walter1987,Abresch1987} are found in this limit.

\subsection{Classification of cylinders}
Theorem~\ref{thm:localIsothermicPlanar} gives the immersion formulas for isothermic surfaces with one generic family ($u$-curves) of planar curvature lines. In this section we investigate when the planar $u$-curves are closed, so that $f$ is a cylinder.

The global properties of the two cases are very different. The rectangular case does not have closed curves, but the rhombic case does.

\begin{lemma}
	Let $f(u,v)$ be an isothermic surface with one family ($u$-curves) of planar curvature lines given by $\gamma(u,\w)$. Then, $f(u + 2\pi,v) = f(u,v)$ if and only if one of the following conditions is satisfied:
	\begin{enumerate}
		\item(Rectangular) 
		\begin{itemize}
		\item general: there exists an $\omega \in \R$, $\omega \neq 0, \frac\pi2$ such that $\vartheta_4'(\omega) = 0$,
		\item limiting $\omega=0$: there exists $\tau \in  \ci\R$ such that $\vartheta_4''(0)=0$,
		\item limiting $\omega=\frac{\pi}{2}$: there exists $\tau \in \ci \R$ such that $\vartheta_3''(0)=0$,
		\end{itemize}
		\item(Rhombic) 
		\begin{itemize}
		\item general: there exists an $\omega \in \R$, $\omega \neq 0, \frac\pi2$ such that $\vartheta_2'(\omega) = 0$,
		\item limiting $\omega=0$: there exists $\tau \in \frac12 + \ci \R$ such that $\vartheta_2''(0)=0$.
		\end{itemize}
	\end{enumerate}
\end{lemma}
\begin{proof}
	From \eqref{eq:fImmersionsPlanar} we see that the $u$-periodicity of $f(u,v)$ is equivalent to the $u$-periodicity of $\gamma(u,\w)$. For both rectangular \eqref{eq:gammaRect} and rhombic \eqref{eq:gammaRhombic} lattices $\gamma(u + 2\pi,\w) = \gamma(u, \w)$ if and only if the exponential term vanishes for all $u + \ci w$.
Therefore, in the general rectangular case $\omega$ must satisfy $\vartheta_4'(\omega) = 0$, and in the rhombic case $\omega$ must satisfy $\vartheta_2'(\omega) = 0$. In the limiting cases governed by Proposition \ref{prop:f_omega=0} the periodicity condition $\hat{\gamma}(u+2\pi,w)=\hat{\gamma}(u,w)$ is equivalent  to vanishing of the corresponding theta constants $\vartheta_k''(0)$.
\end{proof}

The following two technical lemmas show that such parameters $\tau$ and $\omega$ do not exist for rectangular lattices, but exist for some rhombic lattices.

\begin{lemma}
	\label{lem:rectangularNeverClosed}
	For all elliptic curves of rectangular type, i.e., $\tau \in \ci \R$, and $\omega \in (0, \frac\pi2)$, $\vartheta_4'(\omega | \tau) \neq 0$, and $\vartheta_4''(0 | \tau) \neq 0$ and $\vartheta_3''(0 | \tau) \neq 0$. 
	
	Therefore, for the case of elliptic curves of rectangular type, there are no isothermic surfaces with one family of planar curvature lines that are closed.
\end{lemma}
\begin{proof}
	These facts follow immediately from the product representations for theta functions, see \cite{whittaker_watson_1996},
	$$
	\vartheta_4(\omega)=G\prod_{n=1}^\infty (1-2q^{2n-1}\cos 2\omega +q^{4n-2}), \quad G=\prod_{n=1}^\infty (1-q^{2n}), \ 0<q<1.
	$$
	This is a monotonically increasing function of $\omega\in (0, \frac{\pi}{2})$, therefore $\vartheta_4(\omega)>0$.  We also obtain
	$$
	\vartheta_4''(0)=\vartheta_4(0) \sum_{n=1}^\infty \frac{8q^{2n-1}}{(1-q^{2n})^2},
	$$
	which implies $\vartheta_4''(0)>0$. The proof for $\vartheta_3''(0)>0$ is the same.
\end{proof}

\begin{lemma}
\label{lem:rhombic_zeros}
	\begin{itemize}
\item[(i)] There exists a unique elliptic curve of rhombic type, i.e., $\tau_0 = \frac12 + \ci \lambda_0$, such that
\begin{align}
	\label{eq:imctau0}
	\vartheta_2''(0 | \tau_0) = 0\;. \quad \lambda_0
	\approx 0.354729892522\;.
\end{align}
\item[(ii)] For each elliptic curve of rhombic type, i.e., $\tau \in \frac12 + \ci \lambda$, there exists a unique $\omega \in (0,\pi/4)$ such that $\vartheta_2'(\omega | \tau) = \vartheta_2'(-\omega | \tau) = 0$ exactly when $0 < \lambda < \lambda_0$.

The planar curvature lines on the corresponding isothermic surfaces are closed.
\end{itemize}
\end{lemma}

\begin{proof}
(i). Using the product formula
\begin{equation}
\label{eq:theta2_product}
\vartheta_2(\omega)=2Gq^{1/4}\prod_{n=1}^\infty (1+2q^{2n}\cos 2z +q^{4n}), \quad q=e^{i\pi\tau}=ie^{-\pi\lambda}
\end{equation} 
and $\vartheta_2'(0)=0$ we obtain for the second derivative
$$
	\vartheta_2''(0)=\vartheta_2(0) \left(-1-8\sum_{n=1}^\infty \frac{q^{2n}}{(1-q^{2n})^2}\right).
$$
Denoting
$$
p=-q^2= e^{-2\pi\lambda},
$$
the condition $\vartheta_2''(0)=0$ reads as 
\begin{equation}
\label{eq:theta2''=0_1}
\sum_{n=1}^\infty \frac{(-1)^{n+1}p^n}{(1+(-1)^n p^n)^2}=\frac{1}{8}.
\end{equation}
The left  hand side of this equation
$$
\phi(p)=\sum_{n=1}^\infty \phi_n(p), \quad \phi_n(p)=\frac{(-1)^{n+1}p^n}{(1+(-1)^n p^n)^2}
$$
is a monotonically increasing  function. Indeed, by direct computation we obtain
\begin{align*}
\phi_{2k-1}'(p)+\phi_{2k}'(p)= \frac{p^{2k-2}( (2k-1)(1+p^{2k-1})(1+p^{2k})^3 - 2k (1-p^{2k})(1-p^{2k-1})^3)}{(1-p^{2k})^3 (1+p^{2k})^3}.
\end{align*}
It is easy to see that the numerator is positive for any $p\in (0,\frac{1}{2})$, with a little bit more work one can show that it is positive for all $p\in (0,1)$.

Now using $\phi(0)=0, \phi(\frac{1}{2})>\phi_1(\frac{1}{2})+\phi_2(\frac{1}{2})=\frac{46}{25}>\frac{1}{8}$ 
we obtain the existence and uniqueness of a solution $p\in (0,\frac{1}{2})$ of \eqref{eq:theta2''=0_1}.

(ii).
Applying the modular transformation, see \cite{whittaker_watson_1996},
$$
\frac{1}{2}+\ci\lambda=\tau\to -\frac{1}{\tau} \to -\frac{1}{\tau}+1= \frac{1}{2}+\frac{\ci}{4\lambda},
$$
we obtain 
$$
\vartheta_2(z | \frac{1}{2}+\ci\lambda)=\frac{1}{\sqrt{2\lambda}}\exp(-\frac{z^2}{\pi\lambda}) \vartheta_2(-\frac{iz}{2\lambda} | \frac{1}{2}+\frac{\ci}{4\lambda}).
$$
Then $\vartheta_2'(z)=0$ is equivalent to
$$
\frac{4z}{\pi}=i\frac{\vartheta_2'(\hat{z} | \hat{\tau})}{\vartheta_2(\hat{z} | \hat{\tau})}, \quad \hat{z}=\frac{iz}{2\lambda}, \ \hat{\tau}=\frac{1}{2}+\frac{\ci}{4\lambda}.
$$
The product representation \eqref{eq:theta2_product} yields
\begin{equation}
\label{eq:theta2'=0}
\frac{4z}{\pi}=\tanh\frac{z}{2\lambda}+4\sinh\frac{z}{\lambda}\sum_{n=1}^\infty \frac{(-1)^n\hat{p}^n}{1+2(-1)^n\hat{p}^n\cosh\frac{z}{\lambda}+\hat{p}^{2n}}, \quad \hat{p}=e^{-\frac{\pi}{2\lambda}}.
\end{equation} 
Consider the right hand side as a function of $z$:
\begin{equation}
\label{eq:psi}
\psi(z):=\tanh\frac{z}{2\lambda}+4\sinh\frac{z}{\lambda}\chi (z),
\end{equation}
where
$$
\chi(z)=\sum_{n=1}^\infty \chi_n(z), \quad \chi_n(z)=\frac{(-1)^n\hat{p}^n}{1+2(-1)^n\hat{p}^n\cosh\frac{z}{\lambda}+\hat{p}^{2n}}.
$$
Using the representation 
$$
\chi_n(z)=\frac{1}{2(\cosh\frac{z}{\lambda}+(-1)^n\cosh\frac{\pi n}{2\lambda})}
$$
it is easy to check that
$$
\chi_{2k-1}(z)+\chi_{2k}(z)<0 \quad \forall z<\frac{\pi}{2},
$$
which yields $\chi(z)<0$. In the same way one can show that
$ \chi'(z)<0, \chi''(z)<0$ as well.
This, in particular, implies
\begin{equation}
\label{eq:psi_bounded}
\psi(\frac{\pi}{4})<1.
\end{equation}
Further 
$$
\psi'(0)=\frac{1}{2\lambda}(1+8\chi(0)),
$$
and $\chi(0)\to 0$ when $\hat{p}\to 0$. Thus for sufficiently small $\lambda$ we have $\psi'(0)>\frac{4}{\pi}$. This together with \eqref{eq:psi_bounded} implies the existence of a solution 
$\hat{p}_0\in (0,\frac{\pi}{4})$ of \eqref{eq:theta2'=0}.
Differentiating \eqref{eq:psi} we obtain 
$$
\psi''(z)=-\frac{1}{3\lambda^2}\frac{\sinh\frac{z}{2\lambda}}{\cosh^3\frac{z}{2\lambda}}
+4\sinh\frac{z}{\lambda}(\chi''(z)+\frac{1}{\lambda^2}\chi(z)) +
\frac{8}{\lambda}\cosh\frac{z}{\lambda}\chi'(z).
$$
All terms in this expression are negative, so $\psi''(z)<0$. The uniqueness of the solution of \eqref{eq:theta2'=0} then follows from the concavity of $\psi(z)$, see Fig.~\ref{fig:lemma_rhombic_zeros}.

\begin{figure}[tbh!]
	\begin{center}
		\includegraphics[width=0.5\textwidth]{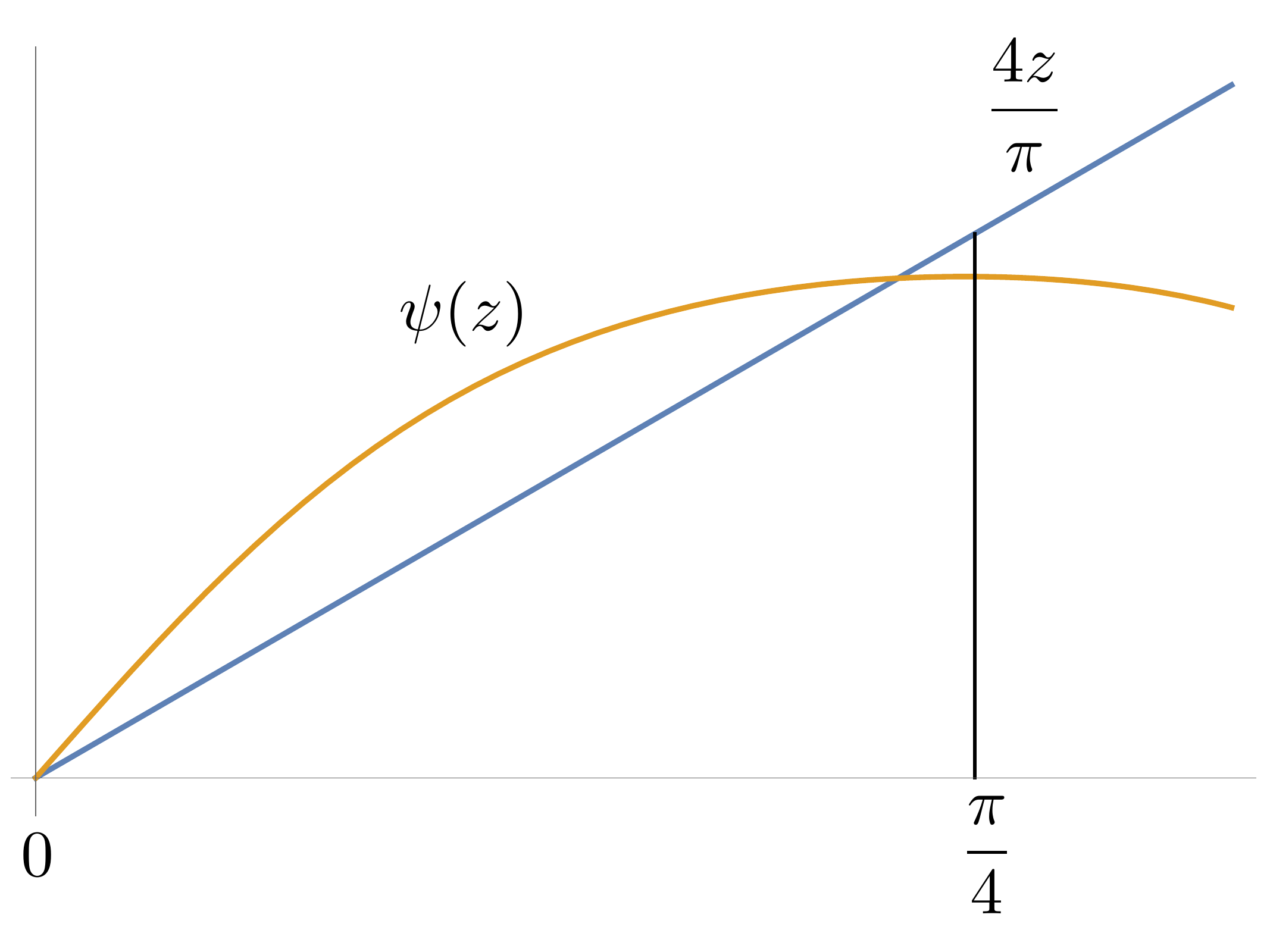}
	\end{center}
	\caption{To the proof of Lemma~\ref{lem:rhombic_zeros}. If $\psi'(0)> \frac{4}{\pi}$ the solution $0<z<\frac{\pi}{4}$ of $\psi(z)=\frac{4z}{\pi}$ exists and is unique since  $\psi(\frac{\pi}{4})<1$ and $\psi$ is concave.}
	\label{fig:lemma_rhombic_zeros}
\end{figure}

When $\psi'(0)=\frac{4}{\pi}$ the graph of the function $\psi(z)$ is tangent to the line $\frac{4z}{\pi}$ at the origin, the solution of $\vartheta'_2(\omega)=0$ goes to 0, and we arrive at the case $\vartheta_2''(0)=0$ described by \eqref{eq:theta2''=0_1}. Thus the same $\lambda_0$ is the unique solution of another equation
$$
\frac{4}{\pi}=\frac{1}{2\lambda}+\frac{4}{\lambda}\sum_{n=1}^\infty \frac{(-1)^n\hat{p}^n}{(1+(-1)^n \hat{p}^n)^2}, \quad \hat{p}=e^{-\frac{\pi}{2\lambda}}.
$$
When $\lambda > 
\lambda_0$, $\psi'(0) < \frac{4}{\pi}$ and so \eqref{eq:theta2'=0} has no real solutions.
\end{proof}

\begin{remark}
It can be checked that the condition $\vartheta_2''(0 | \tau_0)=0$ in terms of the Weierstrass data can be formulated as $\omega_1 e_1+\eta_1=0$. In this form it appeared in \cite[Theorem~3.10]{MR3667223} in its relationship to the spectrum of a Lam\'e operator, where also existence and uniqueness of the corresponding rhombic lattice was shown. 
\end{remark}

To state global results, we give need to midly restrict the allowable set of reparametrization functions. These are the exact conditions for the isothermic surface with one family of planar curvature lines to be smooth (or real analytic, see Remark~\ref{rem:admissibleReparametrization}).

\begin{definition} 	\label{def:admissibleReparametrizationFunction}Fix $\tau \in \frac12 + \ci \R$ such that $0 < \Imc \tau < \Imc \tau_0 \approx 0.3547$. A map $\w: \R \to \R$ is called a $\tau$-\emph{admissible reparametrization function} if
	\begin{enumerate}
		\item $\w: \R \to (0, 2 \pi \Imc \tau)$ is a smooth function and
		\item $\sqrt{1-\w'(\cdot)^2}: \R \to \R$ is also a smooth function.
	\end{enumerate}
\end{definition}

\begin{remark}
	\label{rem:admissibleReparametrization}
	\begin{itemize}
		\item The range of $\w$ is bounded because the closed curve $\gamma(\cdot, \w)$ degenerates for $\w = 0$ and $\w = 2\pi \Imc \tau$.
		
		\item We necessarily have $|\w'(v)| \leq 1$ for all $v \in \R$. Note that the associated square root function $\sqrt{1-\w'(\cdot)^2}$ can change sign at $v_0$ where $\w'(v_0) = 1$.
		
		\item Replacing the word `smooth' with `real analytic' in Definition~\ref{def:admissibleReparametrizationFunction} are the exact conditions for the isothermic surfaces with one family of planar curvature lines to be real analytic.
	\end{itemize}

\end{remark}

Combining Theorem~\ref{thm:localIsothermicPlanar} and the previous two lemmas proves the following global classification.

\begin{theorem}
	\label{thm:planarIsothermicCylinderFormulas}
	Every isothermic cylinder with one generic
	family ($u$-curves) of closed planar curvature lines is given in isothermic parametrization by the following formulas.
	\begin{align}
		f(u,v) &= \Phi^{-1}(v) \gamma(u,\w(v)) \qj \Phi(v), 
		\label{eq:isothermicCylinder}\\
		\gamma(u,\w) &=  -\ci \frac{2 \vartheta _2(\omega)^2}{\vartheta_1^{\prime}(0) \vartheta _1(2 \omega)} \frac{\vartheta_1\left(\frac{1}{2} (u+\ci \w-3 \omega)\right)}{\vartheta_1\left(\frac{1}{2} (u+\ci\w+\omega)\right)},
		\label{eq:gammaClosedCurves} \\
		\Phi'(v) \Phi^{-1}(v) &= \sqrt{1-\w'(v)^2} W_1(\w(v)) \qk, \label{eq:phiPrimePhiInverse}\\
		W_1(\w) &=\ci \frac{\vartheta _1^{\prime}(0)}{2\vartheta _2(\omega)}\frac{\vartheta _2(\omega -\ci \w)}{\vartheta_1(\ci \w)}.
	\end{align}
	The theta functions $\vartheta_i(z | \tau)$ are defined on an elliptic curve of rhombic type spanned by $\pi$ and $\pi \tau$ with $\tau = \frac12 + \ci \R$ satisfying $0 < \Imc \tau < \Imc \tau_0 \approx 0.3547$ (defined by $\vartheta_2''(0 | \tau_0) = 0$). The parameter $\omega$ is the unique \emph{critical} $\omega \in (0, \pi/4)$ satisfying $\vartheta _2^{\prime}(\omega | \tau) = 0$. The curvature line planes are tangent to a cone with apex at the origin. The function $\w(v)$ is a $\tau$-admissible reparametrization function.

	The degenerate case $\omega=0$ when all curvature line planes are tangent to a cylinder is given by immersion formula \eqref{eq:f_omega=0} with $\vartheta_2''(0)=0$ so
\begin{align}
\hat{\gamma}(u,\w)= 
2\ci \frac{\vartheta_2^2 (0) \vartheta_1' (\frac{1}{2}(u+\ci\w))}{\vartheta_1^{'2} (0)  \vartheta_1(\frac{1}{2}(u+\ci\w))}, \quad 
r(\w)=\frac{\vartheta_2(0) \vartheta_2'(\ci\w)}{\vartheta_1'(0) \vartheta_1(\ci\w)}.
\end{align}		
\end{theorem}

\begin{corollary} Let $f(u,v)$ be an isothermic cylinder with one generic family of planar curvature lines, as in Theorem~\ref{thm:planarIsothermicCylinderFormulas}. Then its frame is 
	\begin{align}
		\label{eq:fuIsothermicPlanarClosed}
		f_u(u,v) &= e^{h(u,\w(v))} \Phi(v)^{-1} e^{\ci \sigma(u,\w(v))} \qj \Phi(v), \\
		\label{eq:fvIsothermicPlanarClosed}
		f_v(u,v) &= e^{h(u,\w(v))} \Phi(v)^{-1} \left( \sqrt{1-\w'(v)^2} \, \qi + \w'(v) \, e^{\ci \sigma(u,\w(v))} \qk \right) \Phi(v), \\
		\label{eq:nIsothermicPlanarClosed}
		n(u,v) &= \Phi(v)^{-1} \left( \w'(v) \, \qi - \sqrt{1-\w'(v)^2} \, e^{\ci \sigma(u,\w(v))} \qk \right) \Phi(v),
	\end{align}
	and the family of closed planar curves $\gamma(u,\w)$ satisfies the following.
	\begin{align}
		\label{eq:gammaUClosedCurves}
		\gamma_u(u,\w) &= -\ci \gamma_{\w}(u,\w) = e^{h(u,\w) + \ci \sigma(u,\w)} = -\ci \left(\frac{\vartheta_2\left(\frac{u + \ci \w - \omega}{2}\right)}{\vartheta_1\left(\frac{u + \ci \w + \omega}{2}\right)}\right)^2, \\
		\label{eq:etohClosedCurves}
		e^{h(u,\w)} &= \frac{\vartheta_2\left(\frac{u + \ci \w - \omega}{2}\right)}{\vartheta_1\left(\frac{u + \ci \w + \omega}{2}\right)}\frac{\vartheta_2\left(\frac{u - \ci \w - \omega}{2}\right)}{\vartheta_1\left(\frac{u - \ci \w + \omega}{2}\right)},\\
		\label{eq:etoisigmaClosedCurves}
		e^{\ci \sigma(u,\w)} &= -\ci \frac{\vartheta_2\left(\frac{u + \ci \w - \omega}{2}\right)}{\vartheta_1\left(\frac{u + \ci \w + \omega}{2}\right)}\frac{\vartheta_1\left(\frac{u - \ci \w + \omega}{2}\right)}{\vartheta_2\left(\frac{u - \ci \w - \omega}{2}\right)}.
	\end{align}
\end{corollary}

\begin{corollary}
	The moduli space of isothermic cylinders with one generic family of closed planar curvature lines is parametrized by a real lattice parameter $0 < \Imc \tau < \Imc \tau_0 \approx 0.3547$ \eqref{eq:imctau0} that determines critical $\omega \in (0,\pi/4)$ and a $\tau$-admissible reparametrization function $\w(v)$.
\end{corollary}

\subsection{Global symmetries: sphere inversion and Christoffel duality} 

The explicit formulas in the above classification reveal the following remarkable global geometric properties. Figure~\ref{fig:uCurvePlaneDualAndInverse} illustrates that the spherical inversion and dualization operations described below map each planar curvature line, and therefore the entire surface, onto (minus) itself.

\begin{figure}[tbh]
\centering
\includegraphics[width=0.7\hsize]{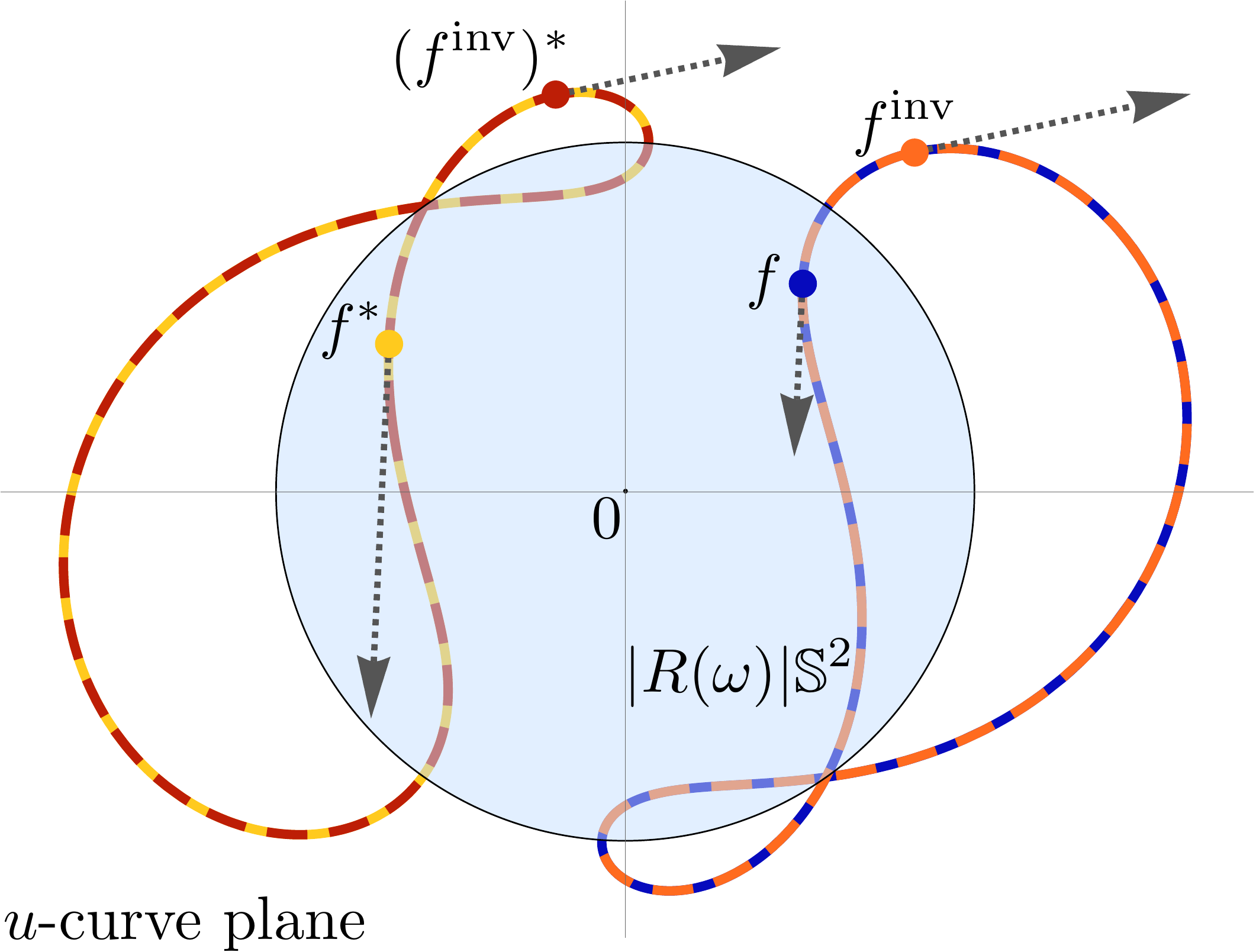}
\caption{A $u$-curve plane of an isothermic surface $f$ with one generic family of closed planar curvature lines. The transformations $(\cdot)^{\textup{inv}}$, inversion in the sphere of radius $| R(\omega) |$ centered at the origin, and $(\cdot)^*$ a Christoffel dualization, are noncommuting involutions that map every closed planar curvature line onto (possibly minus) itself.}
\label{fig:uCurvePlaneDualAndInverse}
\end{figure}

\begin{theorem}
	\label{thm:planarIsothermicCylinderGeometry}
	Every isothermic cylinder $f(u,v)$ with one generic family of closed planar curvature $u$-lines has the following geometric properties:
	\begin{enumerate}
		\item 	 The $v$-curve defined by critical $u = \omega$ lies on a sphere of radius $| R(\omega)|$, where $R(\omega)$ is the real number given by
		\begin{align} R(\omega) = \frac{2 \vartheta _2(\omega)^2}{\vartheta_1^{\prime}(0) \vartheta _1(2 \omega)} \label{eq:radiusOmega}.
		\end{align}
	 Moreover, $f(\omega,v)$ and $f_u(\omega,v)$ are parallel and satisfy
		\begin{align}
			\frac{1}{R(\omega)}f(\omega,v) = -\frac{f_u(\omega,v)}{|f_u(\omega,v)|}. \label{eq:fomegafuParallel}
		\end{align}
		\item Define $f^{\textup{inv}} = -R(\omega)^2 f^{-1}$ as the inversion of $f$ in the sphere of radius $| R(\omega) |$. Then,
		\begin{align}
			\label{eq:fInvUShift}
			f^{\textup{inv}}(u, v) = f(2 \omega-u, v).
		\end{align}
		In other words, this inversion maps $f$ onto itself and is an involution. The plane of each $u$-curve is mapped to itself.
		\item A Christoffel dual $f^*$ of $f$, with $(f^*)_u = \frac{f_u}{| f_u |^2}$ and $(f^*)_v = - \frac{f_v}{| f_v |^2}$, is 
		\begin{align}
			\label{eq:fDualUShift}
			f^*(u,v) = - f(\pi - u, v). 
		\end{align}
		In other words, this duality maps $f$ onto (minus) itself and is an involution. The plane of each $u$-curve is mapped to itself.
	\end{enumerate}
\end{theorem}
\begin{proof}
	We prove each statement in turn.
	\begin{enumerate} 
		\item [\eqref{eq:fomegafuParallel}.] Since $| f | = | \gamma |$, we compute $| \gamma(\omega, \w) |^2$
		using \eqref{eq:gammaClosedCurves}.
		
The conjugation formulas \eqref{eq:theta_conjugation} imply that 
\begin{equation}
\label{eq:gammaConjugation}
\overline{\gamma(u,\w)} =  -\gamma(u,-\w).
\end{equation}
Furthermore, 
		\begin{equation*}
			| \gamma(\omega, \w) |^2 = -\gamma(\omega, \w) \gamma(\omega,-\w) = \left ( \frac{2 \vartheta _2(\omega)^2}{\vartheta_1^{\prime}(0) \vartheta _1(2 \omega)} \right )^2 = R(\omega)^2,
		\end{equation*}
where $R(\omega)$ is given by \eqref{eq:radiusOmega}.
It remains to prove \eqref{eq:fomegafuParallel}. By $\eqref{eq:isothermicCylinder}$, this is equivalent to 
		\begin{align*}
			\frac{1}{R(\omega)} \gamma(\omega, \w(v)) = -\frac{\gamma_u(\omega,\w(v))}{|\gamma_u(\omega,\w(v))|} = -e^{\ci \sigma(u, \w(v))},
		\end{align*}
		which is easily verified using \eqref{eq:gammaClosedCurves} and \eqref{eq:etoisigmaClosedCurves}.		
		\item [\eqref{eq:fInvUShift}.] We find that $f^{\textup{inv}}(u,v) = \Phi^{-1}(v) \gamma^{\textup{inv}}(u,\w(v)) \qj \Phi(v)$, with 
		$\gamma^{\textrm{inv}}(u,\w) := R(\omega)^2 (\overline{\gamma(u,\w)})^{-1}$. 
	From \eqref{eq:gammaConjugation} we have
		\begin{align*}
			\gamma^{\textrm{inv}}(u,\w)  = -R(\omega)^2 (\gamma(u,-\w))^{-1}.
		\end{align*}
		Now using that $\gamma(u,\w) = -\ci R(\omega) \frac{\vartheta_1\left(\frac{1}{2} (u+\ci\w-3\omega)\right)}{\vartheta_1\left(\frac{1}{2} (u+\ci \w-\omega)\right)}$ we obtain $\gamma^{\textrm{inv}}(u,\w) = \gamma(2 \omega-u, \w)$ which implies \eqref{eq:fInvUShift}.	
		\item [\eqref{eq:fDualUShift}.] We need to show $f^*(u,v) := - f(\pi - u, v)$ satisfies $(f^*)_u(u,v) = \frac{f_u(u,v)}{\left | f_u(u,v) \right |^2}$ and $(f^*)_v(u,v) = -\frac{f_v(u,v)}{\left | f_v(u,v) \right |^2}$.
	From \eqref{eq:gammaUClosedCurves}	we have
		\begin{align*}
			\gamma_u(\pi - u, \w) = -\ci \left(\frac{\vartheta _1\left(\frac{1}{2} (-u+\ci\w -\omega)\right)}{\vartheta _2\left(\frac{1}{2} (-u+\ci\w+\omega)\right)}\right)^2.
		\end{align*}
The conjugation formulas \eqref{eq:theta_conjugation} imply
		\begin{align}
			\label{eq:gammaPiShiftInMinusU}
			\gamma_u(\pi - u, \w) = (\overline{\gamma_u(u,\w)})^{-1} =  \frac{\gamma_u(u,\w)}{\left | \gamma_u(u,\w) \right |^2}.
		\end{align}
		Therefore, $(f^*)_u(u,v) = \frac{f_u(u,v)}{\left | f_u(u,v) \right |^2}$. Now, we note that
		\eqref{eq:fvIsothermicPlanar} means
		\begin{equation*}
			\begin{aligned}
			(f^*)_v(u,v) &= -\Phi^{-1} \big( \sqrt{1-\w'(v)^2} \left | \gamma_u(\pi-u,\w(v)) \right | \qi \\&\hspace{2.25cm}+ \w'(v) \gamma_u(\pi-u,\w(v)) \qk \big) \Phi,
			\end{aligned}
		\end{equation*}
		which can be combined with \eqref{eq:gammaPiShiftInMinusU} to prove $(f^*)_v(u,v) = -\frac{f_v(u,v)}{\left | f_v(u,v) \right |^2}$.
	\end{enumerate}
\end{proof}

\begin{remark}
The Christoffel dual of an isothermic torus does not preserve periodicity properties and is usually only simply connected. The remarkable Christoffel symmetry \eqref{eq:fDualUShift} means that if $f$ is a torus, then $f^*$ is a torus. Additionally using the sphere inversion symmetry shows that $(f^{-1})^*$ is also a torus. These topological properties significantly reduce the periodicity conditions to conformally transform $f$ into a Bonnet pair of tori~\cite{short-compact-bonnet}.
\end{remark}

\subsection{Closing the cylinders into tori}
Theorem~\ref{thm:planarIsothermicCylinderFormulas} classified isothermic cylinders with one family ($u$-curves) of closed planar curvature lines, so $f(u,v) = f(u + 2\pi, v)$. Now, we would like to close the cylinders into tori by closing the $v$-curves. The behavior of the $v$-curves depends on the choice of $\tau$-admissible reparametrization function that is periodic $\w(v) = \w( v + \V)$. Either $f(u, v + \V) = f(u,v)$ so the cylinder closes to a torus after one period of $\w(v)$, or a fundamental piece is traced out.

\begin{definition}
	\label{def:isothermicCylinderFromAFundamentalPiece}
	We say $f(u,v) = \Phi(v)^{-1} \gamma(u,\w(v)) \qj \Phi(v)$ is an \emph{isothermic cylinder from a fundamental piece} if
	\begin{enumerate}
		\item $f(u,v)$ is an isothermic cylinder with one generic family ($u$-curves) of closed planar curvature lines as  in Theorem~\ref{thm:planarIsothermicCylinderFormulas} and 
		\item $\w(v)$ has period $\V$, but $f$ is not closed after only one period, i.e., 
		\begin{align*}
			\w(v + \V) = \w(v) \quad \text{ and } \quad	f(u,v + \V) \neq f(u,v). 
		\end{align*}
		The \emph{fundamental piece $\FP$} is the parametrized cylinderical patch
		\begin{align*}
			\FP = \{ f(u,v) \, \big \vert \, u \in [0,2\pi] \text{ and } v \in [0, \V]\}.
		\end{align*}
		\end{enumerate}
The \emph{axis} $A$ and \emph{generating rotation angle} $\angleFP \in [0,\pi]$ are defined by the monodromy matrix of the ODE for $\Phi$ \eqref{eq:phiPrimePhiInverse}:
\begin{equation}
	\label{eq:monodromy}
	\Phi(0)^{-1}\Phi(\V) = \cos (\angleFP/2) + \sin (\angleFP/2) \axisDir.
\end{equation}

\end{definition}

An isothermic cylinder from a fundamental piece $\Pi$ is extended by piecing together congruent copies of $\Pi$ via a fixed rotation about its axis, see Figure~\ref{fig:fundamental-piece-axis}.
\begin{figure}
	\centering
	\includegraphics[width=\hsize]{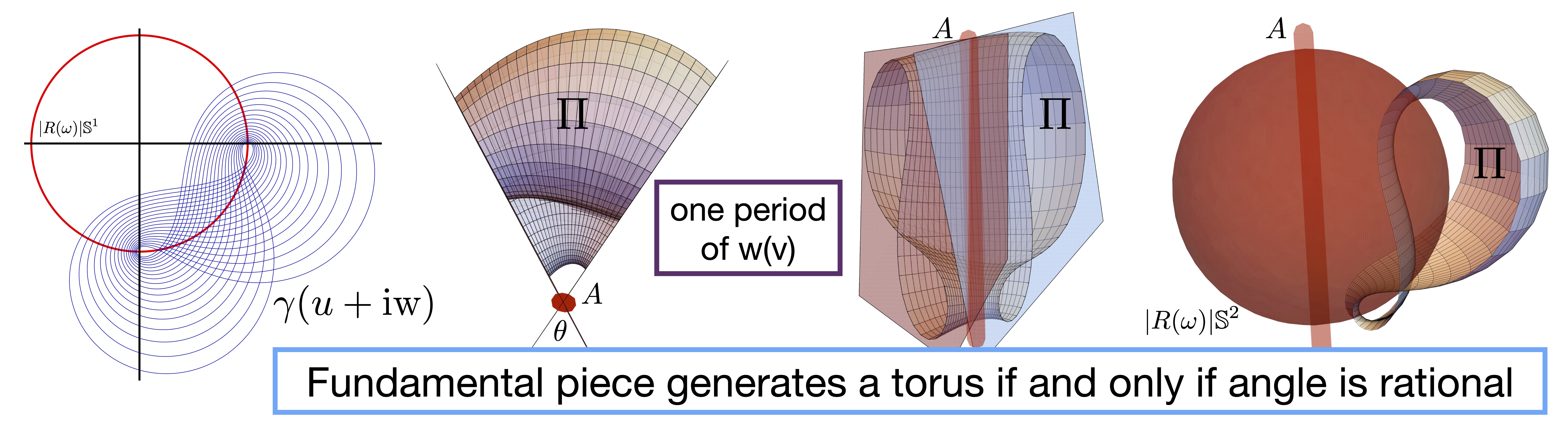}
	\caption{Left: A family of closed area constrained hyperbolic elastica, with its symmetry circle. Middle: Two views of a fundamental piece $\Pi$ with axis $\axisDir$ and angle $\angleFP$ traced out by one period of a periodic $\tau$-admissible reparametrization function $\w(v)$. Right: A third view of the fundamental piece showing the axis and sphere symmetry.}
	\label{fig:fundamental-piece-axis}
\end{figure}

\begin{lemma}
	\label{lem:rotationalSymmetryIsothermicCylinderFromFundamentalPiece}
	Let $f$ be an isothermic cylinder from a fundamental piece $\FP$ with axis $\axisDir$ and angle $\angleFP$. Define the rotation quaternion $R = \cos (\angleFP/2) + \sin (\angleFP/2) \axisDir$. Then for all $1 < k \in \N, u \in [0,2\pi],$ and $v \in [0, \V]$ we have
	\begin{equation}
		\label{eq:rotation_f}
		f(u,v+k \V) = R^{-k} f(u,v) R ^{k}.
	\end{equation}
\end{lemma}
\begin{proof}
	The traversed family of planar $u$-curves only depends on $\w(v)$, therefore $\gamma(u,v) = \gamma(u, v + \V)$, and
	$
	f(u,v+k \V)
	= \Phi(v + k \V)^{-1} \gamma(u,\w(v)) \qj \Phi(v + k \V).
	$	
	The frame $\Phi(v)$ is integrated from the ODE \eqref{eq:phiPrimePhiInverse}. When $\w(v)$ is periodic this ODE has periodic coefficients 
	with monodromy matrix $\Phi(0)^{-1}\Phi(\V)$. In other words, 
	\begin{align}
		\label{eq:phiMonodromyForAllV}
		\Phi(v + k \V) = \Phi(v) \left( \Phi(0)^{-1}\Phi(\V) \right)^{k}.
	\end{align}
	Continuing the calculation from above, we arrive at \eqref{eq:rotation_f}.
	This formula desribes the $k$-times rotation with the axis $\axisDir$ and generating rotation angle $\theta \in [0,\pi]$.
\end{proof}

Thus, closing the isothermic cylinder into a torus is a rationality condition.

\begin{lemma}
	An isothermic cylinder $f$ from a fundamental piece is a torus if and only if the generating rotation angle $\angleFP \in [0,\pi]$ satisfies $k \angleFP \in 2\pi \N $ for some $k \in \N$ so the $v$-period is $k \V$.
\end{lemma}

We summarize into the following theorem.

\begin{theorem}
	\label{thm:rationalityTorusClosing}
	Every isothermic torus $f$ with one generic family of planar curvature lines is given by an isothermic cylinder with closed $u$-curves as in Theorem~\ref{thm:planarIsothermicCylinderFormulas} with $\V$-periodic $\tau$-admissible reparametrization function $\w(v)$ such that either
	\begin{enumerate}
		\item $f$ closes in $v$ after one period of $\w$, i.e., $f(u, v + \V) = f(u, v)$ for all $u,v$,
		\item or $f$ closes in $v$ after $k>1, k \in \N$ periods of $\w$. One period of $\w$ defines an isothermic cylinder from a fundamental piece as in Definition~\ref{def:isothermicCylinderFromAFundamentalPiece} with generating rotation angle $\theta$ that is a rational multiple of $2\pi$.
	\end{enumerate}
\end{theorem}

For case 1. it is unclear how to prove tori exist. For case 2, one proves existence as follows. Each $\V$-periodic $\tau$-admissible reparametrization function $\w(v)$ gives an isothermic cylinder from a fundamental piece with generating rotational angle $\theta$. If $\theta$ is not a rational multiple $2\pi$, perturb the function $\w$.

\section{Isothermic surfaces with one family of planar and one family of spherical curvature lines}
\label{sec:isothermic-planar-u-spherical-v}

Wente's constant mean curvature tori are a prominent example of isothermic tori with one family of planar curvature lines. In fact, each Wente torus has one (non-generic) family of planar curvature lines and one family of spherical curvature lines. Indeed, if a constant mean curvature surface has one family of planar curvature lines then its second family of curvature lines must be spherical.

In contrast, the functional freedom $\w(v)$ determines the second family ($v$-curves) of curvature lines for an isothermic surface with one family ($u$-curves) of curvature lines. In this section, we derive the functional form of the reparamet\-rization function $\w(v)$ that is equivalent to the $v$-curvature lines being spherical. Recall that the planar $u$-curves are governed by an elliptic curve. The functional form equivalent to spherical $v$-curves turns out to be governed by a second elliptic curve, see Remark~\ref{rem:wOfVTwoEllipticCurves}.

Isothermic tori with one generic family of planar and one family of spherical curvature lines are natural generalizations of Wente tori. They may be of independent interest for future research.

\subsection{The axis and cone point of a surface with planar and spherical curvature lines}
We begin by studying a surface $f$ parametrized by curvature lines with one generic family ($u$-curves) of planar curvature lines. Note that we do not assume that $f$ is isothermic, but we retain the generic assumption that the normals to the planes span three dimensional space.

The second family of curvature lines are spherical if for each $u=u_0$ the $v$-curve $f(u_0,v)$ lies on a sphere. Since $f_v(u_0,v)$ lies tangential to the sphere, we have the following in terms of the Gauss map $n(u,v)$, sphere center $Z(u_0)$, and two functions $\sA(u_0)$ and $\sB(u_0)$:
\begin{align}
	\label{eq:sphereCenters}
	Z(u_0) - f(u_0, v) = \sA(u_0) n(u_0, v) + \sB(u_0) \frac{f_u(u_0, v)}{| f_u(u_0, v)|}.
\end{align}
By Joachimsthal's theorem~\cite[p.140 Ex. 59.8]{Eisenhart}, a sphere cuts a surface along a curvature line if and only if the intersection angle is constant. Thus, 
\begin{align}
	\label{eq:sphereRadiusIntersectionAngle}
	\sA(u_0) = R(u_0) \cos\psi(u_0), \quad  \sB(u_0) = R(u_0) \sin\psi(u_0),
\end{align}
where the sphere has radius $R(u_0)$ and the intersection angle is $\psi(u_0)$.

The geometry of surfaces with combinations of planar and spherical curvature lines is a classical topic of differential geometry, see the 1878 book by Enneper~\cite{Enneper1878}. The following result was already well known at that time.

\begin{theorem}
\label{thm:planarSphericalAxisConePoint}
If a surface has one generic family of planar and one family of spherical curvature lines, then the planes intersect in a point and the centers of the spheres lie on a line.
\end{theorem}
\begin{proof}
	We consider $f(u,v)$ parametrized by curvature lines with planar $u$-curves and spherical $v$-curves. Let $m(v)$ be the normals to the planes and $n(u,v)$ the Gauss map of the surface. By Joachimsthal's theorem the angle between the planes and the surface is a function $\delta(v)$ satisfying $\langle n(u,v), m(v) \rangle = \cos\delta(v)$. Moreover, $m(v)$ lies perpendicular to $f_u(u,v)$, so $\langle f_u(u,v), m(v)\rangle = 0$, which implies $\langle f(u,v), m(v)\rangle = c(v)$. Now compute the inner product of the sphere equation \eqref{eq:sphereCenters} with $m(v)$ to find
	\begin{align}
		\label{eq:innerProductSphereCenters}
		\left \langle Z(u), m(v) \right\rangle = \sA(u) \cos\delta(v) + c(v).
	\end{align}
	The proof is in two cases, depending on if $\sA(u)$ is constant.
	
	We first consider the case when $\sA(u)$ is nonconstant.
	Differentiating with respect to $u$, dividing by $\sA'(u)$, and differentiating with respect to $u$ again gives $\left \langle \left(\frac{Z'(u)}{\sA'(u)} \right)', m(v) \right\rangle = 0$.
	The generic assumption that $m(v)$ spans three dimensional space implies $\left(\frac{Z'(u)}{\sA'(u)} \right)' = 0$. Integrating gives constant vectors $A, B\in\R^3$ with
	\begin{align}
		\label{eq:linearSphereCenters}
		Z(u) = \alpha(u) A + B.
	\end{align}
	Therefore, the centers of the spheres lie on a line with direction $A$ passing through $B$. It remains to show that all planes pass through $B$. Substitute \eqref{eq:linearSphereCenters} for $Z(u)$ into \eqref{eq:sphereCenters} to find
	\begin{align*}
		\left \langle \sA(u) A + B - f(u,v), m(v) \right\rangle = \sA(u)\cos\delta(v).
	\end{align*}
	Dividing by $\sA(u)$, differentiating with respect to $u$, and using that $f_u(u,v)$ is perpendicular to $m(v)$, implies 
	\begin{align*}
		\left \langle -B + f(u,v), m(v) \right \rangle = 0,
	\end{align*}
	This equation means that for each $v = v_0$, $f(u,v_0)$ spans the plane through $B$ with normal $m(v_0)$. Therefore, all planes pass through $B$.
	
	If $\sA(u)$ is constant, differentiate \eqref{eq:innerProductSphereCenters} to find $\left \langle Z'(u), m(v) \right\rangle = 0$. Thus, the centers $Z(u)$ collapse to a point, through which all planes pass.
\end{proof}

\subsection{An elliptic curve for the sphericality of the $v$-curves}
We now restrict to the isothermic case. We start with the classification of isothermic surfaces with one generic family ($u$-curves) of planar curvature lines studied in the previous section and derive conditions so that the second family ($v$-curves) of curvature lines is spherical. We use the following algebraic characterization, adapted from \cite[p. 316, eq. 100]{Eisenhart}, to encode sphericality.

\begin{lemma}
An isothermic surface $f(u,v)$ has spherical $v$-curves if and only if there exist $\sA(u),\sB(u)$, functions of $u$ only, with
\begin{align}
\label{eq:sphericalVCharacterization}
0 = e^{h(u,v)} -\sA(u) \bQ(u,v) + \sB(u) h_u(u,v),
\end{align}
where $e^{2h}$ is the metric, and $\bP, \bQ$ are the curvature functions in \eqref{eq:PandQ}.
\end{lemma}
\begin{proof}
The $v$-curves are spherical if and only if \eqref{eq:sphereCenters} holds for all $u_0$, i.e., the centers of the spheres $Z(u)$ satisfy
\begin{align*}
	Z(u) - f(u,v) = \sA(u)n(u,v)+\sB(u)e^{-h(u,v)}f_u(u,v),
\end{align*}
for some $\sA(u)$ and $\sB(u)$ as in \eqref{eq:sphereRadiusIntersectionAngle}.

Differentiating with respect to $v$ yields 
\begin{equation}
\label{eq:vDerivativeOfSphericalCenters}
\begin{aligned}
-f_v(u,v) = \sA(u) n_v(u,v) + \sB(u) & \left( -h_v(u,v) e^{-h(u,v)} f_u(u,v) \right. +  \\ &  + \left. e^{-h(u,v)}f_{uv}(u,v) \right).
\end{aligned}
\end{equation}
Expanding $n_v$ and $f_{uv}$ using \eqref{eq:isothermicFrameEquationsDV} and the definition \eqref{eq:PandQ} for $\bP$ and $\bQ$ gives
\begin{align*}
	0 &= \left(e^{h(u,v)}- \sA(u) \bQ(u,v) + \sB(u) h_u(u,v)\right) \frac{f_v(u,v)}{e^{h(u,v)}},
\end{align*}
which is equivalent to \eqref{eq:sphericalVCharacterization}.

Conversely, if there exist $\sA(u), \sB(u)$ satisfying \eqref{eq:sphericalVCharacterization}, then we can consider the preceding equation. Using the integrability conditions we can rewrite it again as \eqref{eq:vDerivativeOfSphericalCenters}, which can then be integrated with respect to $v$. This recovers the equation for the spherical centers, where $Z(u)$ is the function of integration, so the $v$-curves are spherical. 
\end{proof}

The following technical lemma derives the equations for $\sA(u)$ and $\sB(u)$.

\begin{lemma}
\label{lem:sphericalAlphaBetaExist}
An isothermic surface with one generic family ($u$-curves) of planar curvature lines has its second family ($v$-curves) of curvature lines lying on spheres if and only if $\sA(u), \sB(u)$ from \eqref{eq:sphericalVCharacterization} are
\begin{equation}
	\label{eq:sphericalAlphaBetaFormulas}
	\begin{aligned}
		\sA(u) = \frac{U_1(u)}{U_1(u)\sPhi_0(u)+U(u) \sPhi_2(u)}, \
		\sB(u) = \frac{-\sPhi_2(u)}{U_1(u)\sPhi_0(u)+U(u) \sPhi_2(u)},
	\end{aligned}
\end{equation}
where $U(u)$ and $U_1(u)$ are as in Proposition~\ref{prop:UU1BothCases} and $\sPhi_2(u), \sPhi_1(u), \sPhi_0(u)$ satisfy the system of ordinary differential equations
\begin{align}
\label{eq:sphericalODESystem}
\begin{pmatrix} 
\sPhi_2(u) \\ \sPhi_1(u) \\ \sPhi_0(u)
\end{pmatrix}' = 
\begin{pmatrix}
 0 & -U_1(u) & 0 \\
 2 U(u) & 0 & -2U_1(u) \\
 0 & U(u) & 0 \\
\end{pmatrix} 
\begin{pmatrix} 
\sPhi_2(u) \\ \sPhi_1(u) \\ \sPhi_0(u)
\end{pmatrix}.
\end{align}
\end{lemma}
\begin{proof} To prove that having spherical $v$-curves is compatible with an isothermic surface with planar $u$-curves, we find $\sA(u)$ and $\sB(u)$ in \eqref{eq:sphericalVCharacterization}. The main idea is to rewrite \eqref{eq:sphericalVCharacterization} until one side is a function of $\w$ only. Differentiating with respect to $u$ then yields differential equations that must be satisfied.

We plug $\bQ$ from \eqref{eq:thirdFFPQW} into \eqref{eq:sphericalVCharacterization}, multiply through by $h_\w$, and find
\begin{align*}
0 = e^h h_\w -\sA(u) \bP_\w + \sB(u) h_u h_\w.
\end{align*}
Using \eqref{eq:huhw} 
we can integrate with respect to $\w$. 
\begin{align*}
 \sA(u) \bP = e^h + \sB(u) \frac{h_{u \w}}{h_\w} + \sA(u) \sPhi_1(u)
\end{align*}
for some function $\sPhi_1(u)$ (we multiplied by $\sA(u)$ to ease computation).

Solving for $\bP$ and noting the Riccati equation \eqref{eq:riccati} implies $\frac{h_{u \w}}{h_\w} = U e^h - U_1 e^{-h}$ gives
\begin{align}
\label{eq:bigP}
\bP = \sPhi_2(u) e^{-h} + \sPhi_1(u) + \sPhi_0(u) e^{h},
\end{align}
where we set
\begin{align*}
\sPhi_2(u) = - U_1(u) \frac{\beta(u)}{\alpha(u)} \quad \text{and} \quad \sPhi_0(u) = U(u) \frac{\beta(u)}{\alpha(u)} + \frac{1}{\alpha(u)}.
\end{align*}
To derive differential equations for the $\sPhi_i$, we rewrite \eqref{eq:thirdFFPQW} as
\begin{align}
\label{eq:functionOfWOnly}
\sqrt{1 - \w'(v)^2}= \frac{\bP}{h_\w}
\end{align}
and differentiate with respect to $u$. Since the left hand side is independent of $u$, the resulting equation is equivalent to
\begin{align*}
0 = \bP_u - \bP \frac{h_{u \w}}{h_\w}.
\end{align*}
We compute the right hand side using \eqref{eq:bigP} for $\bP$. Using the Riccati equation \eqref{eq:riccati} again we find that
\begin{align*}
\bP_u &= U \sPhi_0 e^{2h} + \sPhi_0' e^h + (\sPhi_1' + U_1 \sPhi_0 - U \sPhi_2) + \sPhi_2' e^{-h} - U_1 \sPhi_2 e^{-2h}, \\
\bP \frac{h_{u \w}}{h_\w} &= U \sPhi_0 e^{2h} + U \sPhi_1 e^h + (U \sPhi_2 - U_1 \sPhi_0) - U_1 \sPhi_1 e^{-h} - U_1 \sPhi_2 e^{-2h}.
\end{align*}
Their difference only has powers $e^{-h}, e^0, e^h$ and is zero exactly when each coefficient vanishes independently, as $e^h$ depends on both $u,v$, while the coefficients only depend on $u$.  This gives the system \eqref{eq:sphericalODESystem} of three ordinary differential equations for three functions. It has a unique solution for every choice of initial condition.
\end{proof}

\begin{remark} The system \eqref{eq:sphericalODESystem} can be solved explicitly by converting it into a third order differential equation. There are three quantities on equal footing. One is geometric, and is the tangent of the intersection angle
\begin{align}
\tan\psi(u) = \frac{\sB(u)}{\sA(u)}= \frac{-\sPhi_2(u)}{U_1(u)}.
\end{align}
The other two are $\sPhi_1(u)$ and $\frac{\sPhi_0(u)}{U(u)}$. Through long computations to eliminate the other two variables, and by using properties of $U(u)$ and $U_1(u)$, one can find that each of $\frac{-\sPhi_2(u)}{U_1(u)},\sPhi_1(u), \frac{\sPhi_0(u)}{U(u)}$ solve an equation of the form
\begin{align}
	\label{eq:thirdOrderTanIntersectionAngleEquation}
	Y'''(u) - 4 \left( \frac{3}{4} \wp(u-c) \right) Y'(u) - 2 \left(\frac{3}{4} \wp'(u-c) \right) Y(u) = 0,
\end{align}
with $c = \omega, \frac{\pi}{2}, -\omega$, respectively.  Here, $\wp(z)$ is the Weierstrass P function with periods $\pi$ and $\tau \pi$, and the same elliptic curve as $U(u)$ and $U_1(u)$.

The general solution of \eqref{eq:thirdOrderTanIntersectionAngleEquation} depends on three constants $A_0, A_1, A_2$:
\begin{align}
	\label{eq:bigYofu}
		Y(u) = \frac{A_2 \wp^2(\frac{u-c}{2}) + A_1 \wp(\frac{u-c}{2}) + A_0}{\wp'(\frac{u-c}{2})}.
\end{align}
By \cite[Ch. 23, Ex. 17]{whittaker_watson_1996}, $Y(u)$ is the product of two solutions of the $\frac{3}{4}$\mbox{-}Lam\'e equation $y''(u) = \frac{3}{4}\wp\left(u-c\right) y(u)$.

One can therefore integrate the system \eqref{eq:sphericalODESystem} in terms of the intersection angle, which is explicitly given by $\tan{\psi(u)} = Y(u)$ \eqref{eq:bigYofu} with $c = \omega$.
\end{remark}

Remarkably, we do not have to solve the system of differential equations from the previous lemma in order to locally construct surfaces. Instead, we derive the functional form of the reparametrization function $\w(v)$ that is equivalent to having spherical $v$-curves. This depends only on three real parameters, which are the initial data of the differential equations.

\begin{lemma}
	An isothermic surface with one generic family ($u$-curves) of planar curvature lines has its second family ($v$-curves) of curvature lines lying on spheres if and only if there exist three real constants $\sPhi_2(u_0), \sPhi_1(u_0)$, and $\sPhi_0(u_0)$ so that $\w(v)$ is defined by
	\begin{align}
		\label{eq:vPrimeOfWNotYetEllipticIntegral}
		\left(\w'(v)\right)^2 = \frac{h_{\w}(u_0,\w)^2-\bP(u_0,\w)^2}{h_{\w}(u_0,\w)^2} \geq 0.
	\end{align}
	Here, the entry of the third fundamental form at $u = u_0$ is
	\begin{align}
		\label{eq:bPSpherical}
		\bP(u_0,\w) = \sPhi_2(u_0) e^{-h(u_0,\w)} + \sPhi_1(u_0) + \sPhi_0(u_0) e^{h(u_0,\w)}.
	\end{align}
\end{lemma}
\begin{proof}
	Start by solving for
	$\left( \w'(v) \right)^2$
	in \eqref{eq:functionOfWOnly} and noting that the differential equations for the $\sPhi$ variables guarantee that the right hand side of this equation is independent of the choice of $u_0$. We can therefore use the initial conditions $\sPhi_2(u_0), \sPhi_1(u_0), \sPhi_0(u_0) \in \R$ as parameters in \eqref{eq:vPrimeOfWNotYetEllipticIntegral}. To conclude, the necessary condition $|\w'(v)| \leq 1$ is satisfied as $h_{\w}^2 - \bP^2 \geq 0$.
\end{proof}

Through a change of variables, we see that the functional form for $\w'(v)$ \eqref{eq:vPrimeOfWNotYetEllipticIntegral}, characterizing sphericality of the $v$-curves, is governed by an elliptic curve. The coordinate of the elliptic curve is $e^{-h(u_0, \w)}$, which depends on the choice of $u_0$, but $\w'(v)$ is independent of this choice. Due to its geometric significance, we choose $u_0 = \omega$.

\begin{theorem}
	\label{thm:sphericalEllipticCurve}
	The following are equivalent for an isothermic surface with one generic family ($u$-curves) of planar curvature lines with reparametrization function $\w(v)$.
	\begin{enumerate}
		\item The second family ($v$-curves) of curvature lines are spherical.
\item The reparametrization function is given by $\w'(v) = \w'(s) s'(v)$ where
\begin{align}
		\label{eq:sphericalEllipticCurveLocal}
		\w'(s) &= \frac{1}{\sqrt{Q_3(s)}} \quad \text{ and } \quad v'(s) = \frac{\delta}{\sqrt{Q(s)}} \text{ with }\\
		\label{eq:sphericalEllipticCurveFormulaLocal}
		Q(s) &= -(s-s_1)^2(s-s_2)^2 + \delta^{2}Q_3(s) \text{ and }  \\
		\label{eq:uEllipticCurveLocal}
		Q_3(s) &= 2U_1'(\omega) s^3 -U_2(\omega) s^2 - 2 U'(\omega) s - U(\omega)^2
	\end{align}
		for some $0 \neq \delta \in \R$ and either $s_1, s_2 \in \R$ or $s_2 = \overline{s_1} \in \C$. Moreover, $\sqrt{1-\w'(\cdot)^2}$ is given as a signed real valued function by
		\begin{align}
			\label{eq:signedSquareRootWithBigP}
			\sqrt{1-\w'(v)^2} &= \frac{\bP(\omega, e^{-h(\omega, \w(v))})}{h_{\w}(\omega, \w(v))},  \text{ where } \\
			\label{eq:bigPAtOmega}
			\bP(\omega,s) &= s^{-1} \delta^{-1} (s-s_1)(s-s_2).
			\end{align}
	\end{enumerate}
\end{theorem}
\begin{proof}
	Define $s(\w) =  e^{-h(\omega, \w)}$. Exchange parameters $\sPhi_0(\omega), \sPhi_1(\omega), \sPhi_2(\omega)$ in \eqref{eq:bigP} for $\delta, s_1, s_2$ by 
	\begin{align}
		\label{eq:sphericalPhisToDeltaS1S2}
		\sPhi_0(\omega) = \delta^{-1} s_1 s_2,\quad  \sPhi_1(\omega) = -\delta^{-1} (s_1 + s_2),\quad \sPhi_2(\omega) = \delta^{-1},
	\end{align}
	so that \eqref{eq:bigPAtOmega} holds. The third fundamental form equation \eqref{eq:thirdFFPQW} gives \eqref{eq:signedSquareRootWithBigP}. Now, notice that \eqref{eq:hwSquaredU2} implies that $h_{\w}(\omega,s)^2 = s^{-2} Q_3(s)$, since $U_1(\omega)=0$. This gives $\w'(s)$ as in \eqref{eq:sphericalEllipticCurveLocal}, while the form of $v'(s)$ is found by rewriting \eqref{eq:vPrimeOfWNotYetEllipticIntegral} with $u_0 = \omega$. 
\end{proof}

\begin{remark}
	The elliptic curve $\tilde y^2 = Q_3(s)$ governs the planar curvature lines, while the elliptic curve $y^2 = Q(s)$ governs the spherical curvature lines. We refer to these elliptic curves as $Q_3$ and $Q$, respectively.
\end{remark}

\begin{remark}
	\label{rem:wOfVTwoEllipticCurves}
	The variable $s$ is an intermediary between $\w$ and $v$ via
	\begin{align}
		\w'(v) = \w'(s) s'(v).
	\end{align}
	 The functional form governing $\w(v)$, implied by \eqref{eq:sphericalEllipticCurve}, is not often studied.
	 Both $\w$ and $v$ are elliptic integrals of the same variable $s$. Equating the inverse of these elliptic integrals gives the following representation for $\w(v)$.
	 \begin{align}
	 	\label{eq:wOfsISvOfs}
	 	\frac{a_3 \, \wp(\w ; Q_3) + b_3}{c_3 \, \wp(\w ; Q_3) + d_3} = \frac{a \, \wp(v ; \delta^{-2}Q) + b}{c \, \wp(v ; \delta^{-2} Q) + d}.
	 \end{align}
 	On the isothermic surface, $Q_3$ governs the planar $u$-curves via $\w'(s)$, while $Q$ governs the spherical $v$-curves via $v'(s)$.
 	
 	We use a nonstandard notation for the Weierstrass $\wp$ function, motivated by the following result in \cite[Sec. 20.6]{whittaker_watson_1996}. For
 	\begin{align}
 	\label{eq:generalQuartic}
 	\mathcal{Q}(x) = c_4 + 4 c_3 x + 6 c_2 x^2 + 4 c_1 x^3 + c_0 x^4
 	\end{align}
 	the inverse of the elliptic integral $z = \int_{x_0}^x \mathcal{Q}(t)^{-\frac12}\,dt$ (with $\mathcal{Q}(x_0) = 0$) is a rational function of $\wp(z; g_2(\mathcal{Q}), g_3(\mathcal{Q}))$ where the invariants are
 	\begin{align}
 		\label{eq:g2}
 		g_2(\mathcal{Q}) &= c_0 c_4 - 4 c_1 c_3 + 3 c_2^2, \\
 		\label{eq:g3}
 		g_3(\mathcal{Q}) &= c_0 c_2 c_4 + 2 c_1c_2c_3 - c_2^3 - c_0c_3^2-c_1^2c_4.
 	\end{align}
 	We use the notation $\wp(z; \mathcal{Q})$ to emphasize the dependence on $\mathcal{Q}$.
\end{remark}

We summarize the local classification of isothermic surfaces with one generic family of planar and one family of spherical curvature lines.

\begin{corollary}
\label{cor:planarSphericalIsothermicClassification}
The moduli space of isothermic surfaces with one generic family of planar curvature lines and one family of spherical curvature lines has real dimension five.

With the notation of Theorem~\ref{thm:sphericalEllipticCurve} the parameters are: a real elliptic curve $Q_3$ and parameter $\omega \in \R$ (governing the planar $u$-curves), which combine with additional parameters $\delta \neq 0, s_1, s_2$ to determine a second real elliptic curve $Q$ (governing the spherical $v$-curves). There are four cases. Each real elliptic curve is either of rectangular or rhombic type.
\end{corollary}

\subsubsection{Relationship between $Q_3$ and $\tau$}
We clarify how $Q_3$ exactly encodes the lattice spanned by $\pi$ and $\tau \pi$ with nome $q = e^{\ci \pi \tau}$ that governs the planar $u$-curvature lines. The latter elliptic curve arises from combining the Ricatti and harmonic compatibility equations for the isothermic surface. An immediate consequence of these equations is \eqref{eq:hwSquaredU2}, which we restate for convenience: 
	\begin{align*}
		h_{\w}^2 = - U_1(u)^2 e^{-2h} + 2 U_1'(u) e^{-h} - U_2(u) - 2 U'(u)e^h - U(u)^2 e^{2h}. 
	\end{align*}
This equation describes an elliptic curve. Specifically, for each $u_0 \in \R$ define $x(\w) = e^{-h(u_0, \w)}$ and $y(\w) = x'(\w) = - h_\w e^{-h}$ to find that 
	\begin{align}
	\label{eq:planarEllipticCurves}
	 y^2 &= - U_1(u_0)^2 \, x^4 + 2 U_1'(u_0) \, x^3 - U_2(u_0) \, x^2 - 2 U'(u_0) \, x - U(u_0)^2 \\
		&= -U(u_0)^2 (e^{h(u_0, 0)} x - 1)(e^{h(u_0, \frac\pi \ci)} x - 1) (e^{h(u_0, \frac{\pi + \tau \pi}{\ci})} x-1)(e^{h(u_0, \frac{\tau\pi}\ci)} x-1) \nonumber
	\end{align}
for an elliptic curve with lattice spanned by $\pi$ and $\tau \pi$.
	
Strictly speaking this gives a family of elliptic curves, one for each $u_0$. Using 
 \eqref{eq:g2}\eqref{eq:g3} we can compute the lattice Weierstrass invariants $g_2(u_0)$ and $g_3(u_0)$ as a function of $u_0$. Differentiating in $u_0$ yields zero, showing that $g_2$ and $g_3$ are independent of $u_0$. Thus, the apparent family of elliptic curves is in fact different parametrizations of the same elliptic curve. It is the one spanned by $\pi, \tau \pi$ encoded in the system \eqref{eq:complexUSystem} which we recall for convenience is $\frac{U_1''}{U_1} = \frac{U''}{U} = U_2-2 U_1 U$ and $(U_2 + 6 U U_1)' = 0$. We perform the computation to show that $g_2$ is independent of $u_0$. We omit the longer, but analogous, computation showing that $g_3$ is also independent of $u_0$.
	
Applying the coefficients of \eqref{eq:planarEllipticCurves} to the formula \eqref{eq:g2} for $g_2(u_0)$ with the notation of a general quartic \eqref{eq:generalQuartic} yields:
	\begin{align}
		g_2(u_0) = U_1(u_0)^2 U(u_0)^2 + U'(u_0) U_1'(u_0) + \frac{1}{12} U_2(u_0)^2.
	\end{align}
	Differentiating with respect to $u_0$ and collecting terms yields 
	\begin{align}
		g_2' = U_1' U (2 U U_1 + U''/U) + U' U_1 (2 U U_1 + U_1'' / U_1) + \frac{1}{6}U_2 U_2'.
	\end{align}
	Now, using the system \eqref{eq:complexUSystem} yields zero:
	\begin{align}
		g_2' = U_2 (U_1' U + U' U_1 + 1/6 (6 U U_1)') = 0.
	\end{align}

Choosing $u_0 = \omega$ and noting that $e^{h(\omega, \frac\pi \ci)} = 0$ and $U_1(\omega) = 0$ yields the following.

\begin{proposition}
	\label{prop:uLineEllipticCurve}
	The real elliptic curve with lattice spanned by $\pi, \tau \pi$ governing the planar $u$-curves is given by $\tilde y^2 = Q_3(s)$ \eqref{eq:uEllipticCurveLocal} with coordinate $s = e^{-h(\omega, \w)}$.
	Moreover, we have the factorization
	\begin{align}
		\label{eq:Q3factorization}
		Q_3(s) = U(\omega)^2(e^{h(\omega, 0)} x - 1)(e^{h(\omega, \frac{\pi + \tau \pi}{\ci})} x-1)(e^{h(\omega, \frac{\tau\pi}\ci)} x-1).
	\end{align}
\end{proposition}
\section{Isothermic cylinders with one generic family of closed planar curvature lines whose second family of curvature lines are spherical}
\label{sec:isothermic-planar-u-spherical-v-global}
\subsection{Classification in terms of two elliptic curves}
We summarize the key results of the local theory into the setting of closed generic $u$-curvature lines.
\begin{corollary}
	\label{cor:cylinderEllipticCurves}
Let $f(u,v)$ be an isothermic cylinder with one generic family ($u$-curves) of closed planar curvature lines, as in Theorem~\ref{thm:planarIsothermicCylinderFormulas}, with rhombic lattice spanned by $\pi, \tau \pi$, critical parameter $\omega$ satisfying $\vartheta_2'(\omega)=0$, and $\tau$-admissible reparametrization function $\w(v)$. Then,
\begin{enumerate}
	\item The function $s(\w) = e^{-h(\omega, \w)}$ satisfies
	\begin{align}
		s'(\w)^2 &= Q_3(s), \label{eq:uEllipticCurve} \text{ where }\\
		\label{eq:criticalQ3factorized}
		Q_3(s) &= \frac{\vartheta _1^{\prime }(0)^2}{\vartheta_2(\omega)^2} \left(s-\frac{\vartheta _1(\omega)^2}{\vartheta _2(0)^2}\right) \left(s-\frac{\vartheta_3(\omega)^2}{\vartheta _4(0)^2}\right) \left(s-\frac{\vartheta _4(\omega)^2}{\vartheta _3(0)^2}\right) \\
		&= 2U_1'(\omega) s^3 -U_2(\omega) s^2 - 2 U'(\omega) s - U(\omega)^2,
		\label{eq:criticalQ3}
	\end{align}
	and
	\begin{align}
		\label{eq:bigUAtOmegaAndRadius}
		\quad U(\omega) &= -\frac12  \frac{\vartheta_1'(0)}{\vartheta_2(\omega)^2} \vartheta_1(2\omega) = -R(\omega)^{-1} , \quad U_1(\omega) = 0,\\
		U'(\omega) &= -\frac12 \frac{\vartheta_1'(0)}{\vartheta_2(\omega)^2} \vartheta_1'(2\omega) , \quad			U_1'(\omega) = \frac12 \frac{\vartheta_1'(0)^2}{\vartheta_2(\omega)^2}, \\
		U_2(\omega) &= \frac{\vartheta _1^{\prime }(0)^2}{\vartheta _2(\omega)^2} \left(\frac{\vartheta _1(\omega)^2}{\vartheta _2(0)^2}+\frac{\vartheta _4(\omega)^2}{\vartheta_3(0)^2}+\frac{\vartheta _3(\omega)^2}{\vartheta _4(0)^2}\right).
	\end{align}
	Moreover, the elliptic curve \eqref{eq:uEllipticCurve} has rhombic lattice $\pi, \tau \pi$.
		\item The $v$-curvature lines are spherical if and only if $\w'(v) = \w'(s) s'(v)$ where
		\begin{align}
			\label{eq:sphericalEllipticCurve}
			\w'(s) &= \frac{1}{\sqrt{Q_3(s)}} \quad \text{ and } \quad v'(s) = \frac{\delta}{\sqrt{Q(s)}}, \text{ with }\\
			\label{eq:sphericalEllipticCurveFormula}
			Q(s) &= -(s-s_1)^2(s-s_2)^2 + \delta^{2}Q_3(s)
		\end{align}
		for some $0 \neq \delta \in \R$ and either $s_1, s_2 \in \R$ or $s_2 = \overline{s_1} \in \C$.
		In this case, $\sqrt{1-\w'(\cdot)^2}$ as a signed real-valued function is
		\begin{align}
			\label{eq:signedSquareRoot}
			\sqrt{1-\w'(v)^2} &= \frac{\delta^{-1} \left(s(\w(v))-s_1\right)\left(s(\w(v))-s_2\right)}{s'(\w(v))},  \text{ where } \\
			s(\w) &= e^{-h(\omega, \w)} = \frac{\vartheta _2(\omega)^2}{\vartheta _2(0)^2} \left(\frac{\vartheta _1(\omega)^2}{\vartheta _2(\omega)^2}-\frac{\vartheta _1\left(\frac{\ci \w}{2}\right)^2}{\vartheta _2\left(\frac{\ci \w}{2}\right)^2}\right).
		\end{align}
	\end{enumerate}
\end{corollary}
\begin{proof}
	The real elliptic curve for the $u$-lines follows from Proposition~\ref{prop:uLineEllipticCurve}. Most of the coefficients of $Q_3$ are found by using the formulas for $U, U_1$ \eqref{eq:rhombicU}, \eqref{eq:rhombicU1} with critical parameter $\omega$ satisfying $\theta_2'(\omega)=0$. To find $U_2(\omega)$ without computing second derivatives of theta functions, one first derives \eqref{eq:criticalQ3factorized} by combining the factorizations of the metric \eqref{eq:rhombicMetricFactorized} and $Q_3$ \eqref{eq:Q3factorization} with critical $\omega$. Then, one equates this expression with \eqref{eq:criticalQ3} which is just a restatement of \eqref{eq:uEllipticCurveLocal}.
	
	The equivalent formulation of spherical $v$-curvature lines follows from Theorem~\ref{thm:sphericalEllipticCurve}. To find \eqref{eq:signedSquareRoot} we 
	use \eqref{eq:signedSquareRootWithBigP} and compute $h_{\w}(\omega, \w)$ as minus the logarithmic derivative of $e^{-h(\omega, \w)}$ from \eqref{eq:rhombicMetricFactorized}.
	\end{proof}

To close the isothermic cylinder into a torus it is necessary to consider periodic $\tau$-admissible reparametrization functions $\w(v)$. Each choice either closes the torus after one period or gives rise to a fundamental piece with axis $\axisDir$ and angle $\angleFP$, recall Theorem~\ref{thm:rationalityTorusClosing}.

\subsection{A local formula for the axis}
In general, the axis $\axisDir$ of an isothermic cylinder from a fundamental piece (recall Definition~\ref{def:isothermicCylinderFromAFundamentalPiece}) is a global object, found by integrating the ODE for $\Phi(v)$ \eqref{eq:phiPrimePhiInverse} over one period of the reparametrization function $\w(v)$. In this section we prove that if the second family of curvature lines are spherical, however, one can derive a local formula for the axis, see Proposition~\ref{prop:bigZPrimeAtOmegaMovingFrameCoefficients}. This provides analytic control over the axis, which is important to torus closing conditions. For example, this axis formula is a key step to proving the existence of compact Bonnet pairs as conformal transformations of isothermic tori with one generic family of planar curvature lines. One first proves existence of isothermic tori with the additional constraint that the second family of curvature lines are spherical and then perturbs away from this constraint~\cite{short-compact-bonnet}.

By Theorem~\ref{thm:planarSphericalAxisConePoint}, the centers $Z(u)$ of the spherical curvature line spheres are collinear and span a line. The symmetry induced by the fundamental piece means that the axis $\axisDir$ agrees with this line. The symmetry sphere at $u = \omega$ is centered at the origin, see \eqref{eq:fomegafuParallel}, so we study $Z'(\omega)$ to obtain a direction. Thus,
\begin{equation}
	\label{eq:A-Z}
	\pm \axisDir = \frac{Z'(\omega)}{|Z'(\omega)|}.
\end{equation} 

The centers $Z(u)$ of the spherical curvature line spheres satisfy \eqref{eq:sphereCenters}
where $\alpha(u),\beta(u)$ are given by \eqref{eq:sphericalAlphaBetaFormulas}. Differentiating gives
\begin{align*}
	\begin{aligned}
		Z'(u) - f_u(u,v) = \sA'(u) n(u,v) & + \sA(u) n_u(u,v) \\ & + \sB'(u) \left(e^{-h(u,v)}f_u(u,v)\right) + \sB(u) \left( e^{-h(u,v)}f_u(u,v) \right)_u.
	\end{aligned}
\end{align*}
In isothermic coordinates $f_{uu} = h_u f_u - h_v f_v + \bP e^h n$, so $\left(e^{-h(u,v)}f_u(u,v)\right)_u = -h_v \left(e^{-h(u,v)}f_v(u,v)\right) + \bP n$. Thus,
\begin{align*}
	\begin{aligned}
		Z'(u) = (\sA'(u)& + \sB(u) \bP(u,v)) n(u,v) + \sA(u) n_u(u,v) + \\ 
		&+ (\sB'(u) + e^{h(u,v)}) \left(e^{-h(u,v)}f_u(u,v)\right) - \sB(u) h_v(u,v) \left(e^{-h(u,v)}f_v(u,v)\right).
	\end{aligned}
\end{align*}
To evaluate this expression at $ u = \omega $ we prove the following.
\begin{align}
	\label{eq:alphaAlphaPrimeAtOmega}
	\sA(\omega) = 0, \qquad \sA'(\omega) &= -\delta R(\omega) U_1'(\omega),\\
	\sB(\omega) = R(\omega), \qquad \sB'(\omega) &= R(\omega)^2 \left( U'(\omega) + s_1 s_2 U_1'(\omega)\right).
	\label{eq:betaBetaPrimeAtOmega}
\end{align}
The values are derived from \eqref{eq:sphericalAlphaBetaFormulas} for $\sA(u)$ and $\sB(u)$. For $\sA(\omega), \sB(\omega)$ use that $U_1(\omega) = 0$ and $R(\omega) = -U(\omega)^{-1}$. For $\sA'(\omega), \sB'(\omega)$ we differentiate and again use $U_1(\omega) = 0$ and $R(\omega)= -U(\omega)^{-1}$ to find $\alpha'(\omega) = -R(\omega)\frac{U_1'(\omega)}{\sPhi_2(\omega)}$ and $\beta'(\omega) = R(\omega)^2 \left(U'(\omega) + \frac{\sPhi_0(\omega)}{\sPhi_2(\omega)} U_1'(\omega) \right)$. Finally, use $\sPhi_2 = \delta^{-1}$ and $\sPhi_0 = \delta^{-1} s_1 s_2$ as given by \eqref{eq:sphericalPhisToDeltaS1S2}.

Plugging in $u = \omega$ gives
\begin{align}
	\label{eq:bigZPrimeAtOmega}
	\begin{aligned}
		Z'(\omega) = (\sB'(\omega) & + e^{h(\omega,v)})\left(e^{-h(\omega,v)}f_u(\omega,v)\right) - \sB(\omega) h_v(\omega,v) \left(e^{-h(\omega,v)}f_v(\omega,v)\right) + \\
		& + (\sA'(\omega) + \sB(\omega) \bP(\omega,v)) n(\omega,v).
	\end{aligned}
\end{align}

\begin{proposition}
	\label{prop:bigZPrimeAtOmegaMovingFrameCoefficients}
	Let $f(u,v)$ be an isothermic cylinder from a fundamental piece with axis $\axisDir$ and whose second family of curvature lines are spherical. Then, the centers $Z(u)$ of the spherical curvature line spheres are collinear and
	\begin{align}
	\pm \axisDir = \frac{Z'(\omega)}{|Z'(\omega)|} &= 	z_1(s(v))\left(e^{-h(\omega,v)}f_u(\omega,v)\right) + z_2(s(v))\left(e^{-h(\omega,v)}f_v(\omega,v)\right) + z_3(s(v)) n(\omega, v),
	\end{align}
	where
	\begin{align}
		z_1(s) &= \frac{s^{-1}}{|Z'(\omega)|} \left(1 + s R(\omega)^2 \left( U'(\omega) + s_1 s_2 U_1'(\omega)\right) \right),\\
		z_2(s) &= \frac{s^{-1}R(\omega)\delta^{-1}}{|Z'(\omega)|} \sqrt{Q(s)},\\
		z_3(s) &= \frac{s^{-1}R(\omega)\delta^{-1}}{|Z'(\omega)|} \left(-\delta^2 U_1'(\omega) s + (s-s_1)(s-s_2) \right),
	\end{align}
	and
	\begin{align}
		\label{eq:bigZPrimeNormSquared}
		\begin{aligned}
			|Z'(\omega)|^2 =  R(\omega)^2 &\left( 2(s_1 + s_2)U_1'(\omega) + \delta^2 U_1'(\omega)^2 - U_2(\omega)\right) + \\ & + R(\omega)^4 \left( U'(\omega) + s_1 s_2 U_1'(\omega)\right)^2.
		\end{aligned}
	\end{align}
\end{proposition}
\begin{proof}	
	The sphere centers are collinear by Theorem~\ref{thm:planarSphericalAxisConePoint} and $\pm \axisDir = \frac{Z'(\omega)}{|Z'(\omega)|}$ from \eqref{eq:A-Z}. The coefficients from \eqref{eq:bigZPrimeAtOmega} can be converted into the expressions for $z_1,z_3$ using \eqref{eq:alphaAlphaPrimeAtOmega}, \eqref{eq:betaBetaPrimeAtOmega}, and \eqref{eq:bigPAtOmega} for $\bP(\omega,s)$. The equation for $z_2$ additionally uses that, at $u=\omega$, we have $s = e^{-h}$ and $s'(\w) = -h_{\w} s$, which implies  $-h_v = -h_{\w} \w'(s)s'(v) = s ^{-1} s'(v) =s^{-1} \delta^{-1}\sqrt{Q(s)}$.
	It remains to find the normalization factor $|Z'(\omega)|^2$. We have  
	\begin{align*}
		|Z'(\omega)|^2 &= |Z'(\omega)|^2 (z_1(s)^2 + z_2(s)^2 + z_3(s)^2).
	\end{align*}
	This expression is simultaneously a Laurent polynomial in $s=e^{-h(\omega,\w)}$ and a constant independent of $\w$. Therefore, it is equal to its constant term as a Laurent polynomial in $s$, which can be computed as \eqref{eq:bigZPrimeNormSquared}. 
\end{proof}

\begin{remark}
	In~\cite{short-compact-bonnet} we construct isothermic tori with planar and spherical curvature lines, together with additional constraints to close the Bonnet pair tori. The axis formulas in Proposition~\ref{prop:bigZPrimeAtOmegaMovingFrameCoefficients} lead to an explicit formula for the fundamental piece angle $\angleFP$. An example with 3-fold symmetry is shown in Figure~\ref{fig:planarSpherical3FoldFundamentalPiece}.
\end{remark}

\begin{figure}[tbh!]
	\begin{center}
		\includegraphics[width=0.7\hsize]{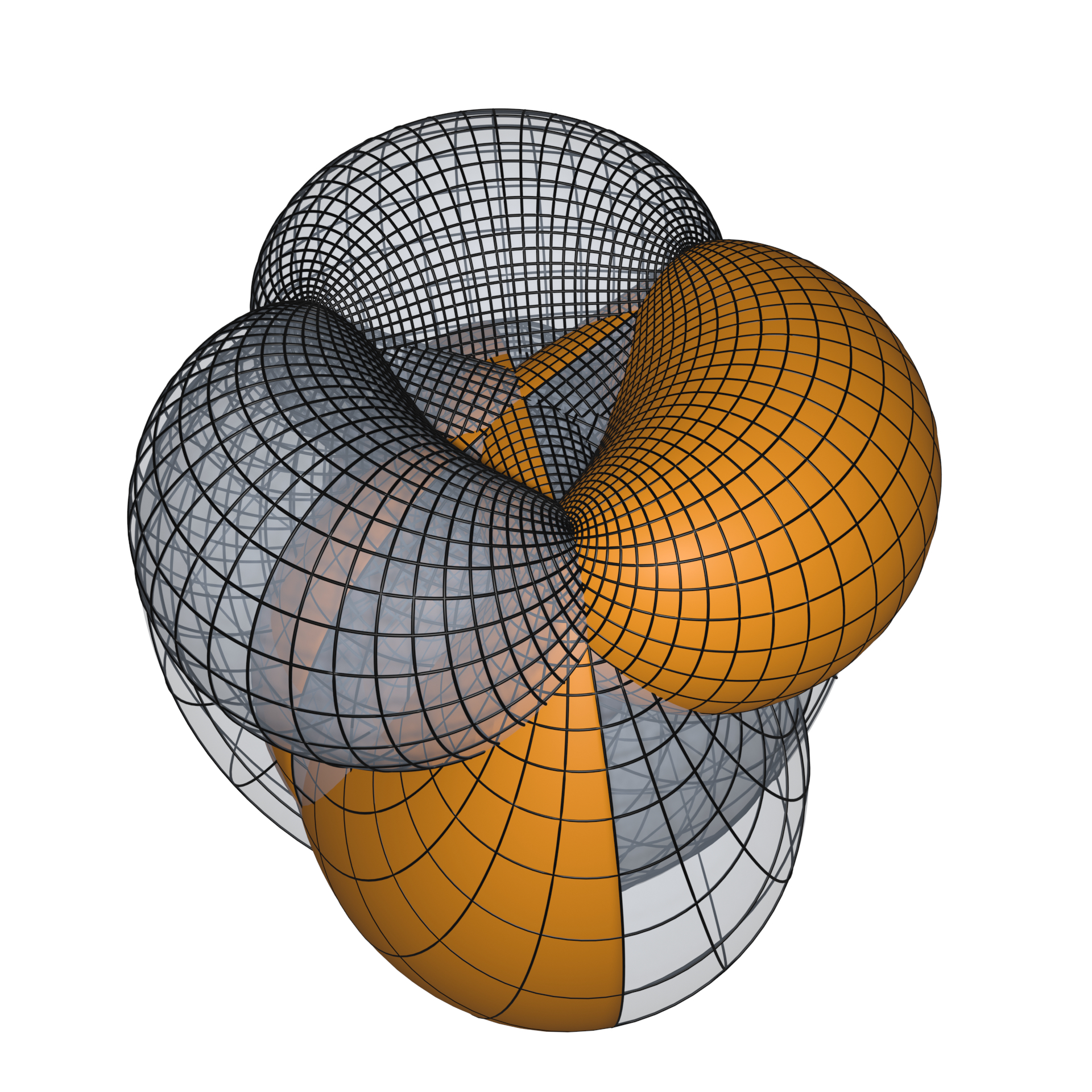}
	\end{center}
	\caption{An isothermic torus with planar and spherical curvature lines constructed in~\cite{short-compact-bonnet}. The fundamental piece with angle $\angleFP = 120^\circ$ rotates about its axis to close the torus with 3-fold symmetry.}
	\label{fig:planarSpherical3FoldFundamentalPiece}
\end{figure}

\footnotesize
\bibliographystyle{abbrv}
\bibliography{short-isothermic-planar}

\end{document}